\documentclass[11pt]{article}
\usepackage{amsmath, latexsym, amsfonts, amssymb, amsthm, amscd}
\usepackage{mathtools} 
\usepackage[left=2.2cm, right=2.2cm, top=2.5cm, bottom=2.5cm]{geometry}
\usepackage{upquote} 
\usepackage{color}
\usepackage{setspace} 
\usepackage[labelfont=bf]{caption} 
\usepackage{stmaryrd} 
\usepackage[title]{appendix}
\usepackage{enumitem} 

\usepackage{kpfonts} 
\usepackage{dsfont} 

\usepackage[english]{babel} 

\usepackage[utf8]{inputenc}
\usepackage[T1]{fontenc}

\usepackage{bm} 

\usepackage{xcolor}
\definecolor{Maroon}{HTML}{ad2231}
\definecolor{webgreen}{HTML}{008000}
\usepackage{hyperref}
\hypersetup{colorlinks, breaklinks, urlcolor=Maroon, linkcolor=Maroon, citecolor=webgreen} 

\allowdisplaybreaks

\usepackage{tikz}
\usetikzlibrary{calc}
\usepackage{pgfplots}

\newtheorem{theorem}{Theorem}[section]
\newtheorem{corollary}[theorem]{Corollary}
\newtheorem{proposition}[theorem]{Proposition}
\newtheorem{lemma}[theorem]{Lemma}

\newtheorem{definition}[theorem]{Definition}

\theoremstyle{definition}

\numberwithin{equation}{section}

\usepackage{chngcntr}
\usepackage{apptools}
\AtAppendix{\counterwithin{lemma}{section}}

\theoremstyle{definition}

\newtheorem*{general*}{A general remark}
\newtheorem*{assumption*}{Assumption}

\begin{document}
\title{Convergence of the pruning processes of stable Galton-Watson trees}
\author{Gabriel Berzunza Ojeda\footnote{ {\sc Department of Mathematical Sciences, University of Liverpool. Liverpool, United Kingdom.} E-mail: gabriel.berzunza-ojeda@liverpool.ac.uk}\,\,
    and
   Anita Winter\footnote{ {\sc Fakult\"at f\"ur Mathematik, Universit\"at Duisburg-Essen. Essen, Germany.} E-mail: anita.winter@uni-due.de} \\ \vspace*{10mm}
}
\maketitle

\vspace{0.1in}

\begin{abstract}
Since the work of Aldous and Pitman (1998) \cite{AldousPitman1998b}, several authors have studied the pruning processes of Galton-Watson trees and their continuous analogue L\'evy trees. In \cite{LohrVoisinWinter2015}, L\"ohr, Voisin and Winter (2015) introduced the space of bi-measure $\mathbb{R}$-trees equipped with the so-called {\em leaf sampling weak vague topology} which allows them to unify the discrete and the continuous picture by considering them as instances of the same Feller-continuous Markov process with different initial conditions. Moreover, the authors show that these so-called pruning processes converge in the Skorokhod space of c\`adl\`ag paths with values in the space of bi-measure $\mathbb{R}$-trees, whenever the initial bi-measure $\mathbb{R}$-trees converge.

In this paper we provide an application to the above principle by verifying that a sequence of suitably rescaled critical conditioned Galton-Watson trees whose offspring distributions lie in the domain of attraction of a stable law of index $\alpha \in (1,2]$ converge to the $\alpha$-stable L\'evy-tree in the leaf-sampling weak vague topology.  
\end{abstract}


\noindent {\sc Subject Classes}: 60F05, 05C05, 05C10


\section{Introduction} \label{Sec1}

Aldous and Pitman \cite{AldousPitman1998b} initiated the study of a tree-valued process obtained by performing dynamical bond-percolation on Galton-Watson trees (GW-trees). They considered a GW-tree $\mathbf{t}_{1}$, and for $u \in [0,1]$, they let $\mathbf{t}_{u}$ be the subtree of $\mathbf{t}_{1}$ containing its root, obtained by removing (or pruning) each edge with probability $1-u$. Aldous and Pitman \cite{AldousPitman1998b} showed that one can couple the above dynamics for several $u \in [0,1]$ in such a way that $\mathbf{t}_{u^{\prime}}$ is a rooted subtree of $\mathbf{t}_{u}$ for $u^{\prime} \leq u$. This pruning procedure on the edges of a GW-tree gives rise to a non-decreasing tree-valued Markov process that they called the pruning process.

The $\alpha$-stable L\'evy tree $\mathcal{T}_{\alpha} = (\mathcal{T}_{\alpha}, r_{\alpha}, \rho_{\alpha}, \mu_{\alpha})$, for $\alpha \in (1,2]$, is a rooted measure $\mathbb{R}$-tree that arises as the scaling limit of $\alpha$-stable ${\rm GW}$-trees (i.e.\ critical GW-trees whose offspring distributions lie in the domain of attraction of an $\alpha$-stable law, conditioned to have a fixed total progeny $N \in \mathbb{N}$). More precisely, let $(\mathbf{t}_{N}, r_{N}^{\text{gr}}, \rho_{N}, \mu_{N}^{\rm nod})$ be an $\alpha$-stable ${\rm GW}$-tree, viewed as a rooted metric measure tree, i.e.\ $\mathbf{t}_{N}$ is identified as its set of $N$ vertices, $r_{N}^{\text{gr}}$ is the graph-distance on $\mathbf{t}_{N}$, $\rho_{N}$ is its root and $\mu_{N}^{\rm nod}$ is the uniform measure on the set of vertices of $\mathbf{t}_{N}$. In particular, by well-known theorems of Aldous \cite{Al1993} and Duquesne \cite{Du2003}, there exists a slowly varying function $\ell: \mathbb{R}_{+} \rightarrow \mathbb{R}_{+}$ such that, for $a_{N} \coloneqq \ell(N) N^{\frac{1}{\alpha} - 1}$,
\begin{equation} \label{eq1}
(\mathbf{t}_{N}, a_{N} \cdot r_{N}^{\text{gr}}, \rho_{N}, \mu_{N}^{\rm nod})\rightarrow (\mathcal{T}_{\alpha}, r_{\alpha}, \rho_{\alpha}, \mu_{\alpha}), \quad \text{in distribution}, \quad \text{as} \quad N \rightarrow \infty,
\end{equation}
\noindent on the space of (equivalence classes) of rooted metric measure spaces equipped with the convergence of the {\em Gromov-Hausdorff-weak topology} (see \cite{AthreyaLoehrWinter2016}).
The Brownian CRT corresponds to $\alpha = 2$, and in particular, if the offspring distribution of $\mathbf{t}_{N}$ has finite variance $\sigma^{2}$ then one could take $\ell \equiv \sigma$.

This result motivated several authors to consider analogues cutting-down procedures and thus pruning processes, of $\alpha$-stable L\'evy trees (and more generally L\'evy trees). For example, Aldous and Pitman \cite{Aldous1998a}, Abraham and Serlet \cite{AS2002} studied the cutting-down of the Brownian CRT along its skeleton according to a Poisson point process of cuts with intensity ${\rm d} t \otimes \lambda_{2}$ on $[0, \infty) \times \mathcal{T}_{2}$, where $\lambda_{2}$ is the so-called length measure (this corresponds to bond-percolation in the discrete setting). In a similar spirit, Miermont \cite{Miermont2005} studied the cutting-down of the $\alpha$-stable L\'evy tree along its branching points in a Poisson manner. Miermont \cite{Miermont2005}, Abraham and Delmas \cite{AD2008} also studied the same pruning procedure in L\'evy trees without a Brownian component. More generally, \cite{ADV2010, AD2012} studied a mixture of the two previous types of cutting-down procedures on L\'evy trees. The discrete analogue of pruning a L\'evy tree along its branching points was described by Abraham, Delmas, and He \cite{ADHe2012}. In \cite{ADHe2012}, the authors considered dynamic site-percolation on GW-trees where each branching vertex of the initial tree is removed (or cut) independently with probability $1 - u^{n-1}$, for some $u \in [0,1]$, where $n$ represents the number of children of the vertex.

In \cite{LohrVoisinWinter2015}, L\"ohr, Voisin and Winter provided a new state space together with a new topology which allows them to unify the discrete and the continuous picture by considering them as instances of the same Feller-continuous Markov process, THE pruning process, with different initial conditions. They constructed this process on the space of so-called bi-measure $\mathbb{R}$-trees (i.e., rooted measure $\mathbb{R}$-trees $(\mathcal{T}, r, \rho, \mu)$ with an additional so-called pruning measure) equipped with the leaf-sampling weak vague topology (LWV-topology); see Definition \ref{Def1}. Moreover, they characterized this process analytically via its Markovian generator.

L\"ohr, Voisin and Winter \cite{LohrVoisinWinter2015} also established convergence of sequences of pruning processes, in the Skorokhod space of c\`adl\`ag paths with values in the space of bi-measure $\mathbb{R}$-tree, whenever the initial bi-measure $\mathbb{R}$-trees converge for the LWV-topology. In particular, and as a consequence of the convergence in (\ref{eq1}), the authors in \cite[Section 4]{LohrVoisinWinter2015} showed that the pruning process of $\alpha$-stable L\'evy trees, with the pruning measure being their underlying length measure, is the scaling limit of dynamic bond-percolation on $\alpha$-stable GW-trees (i.e.\ Aldous and Pitman \cite{AldousPitman1998b} pruning process). Thus, it is natural to expect that the more general pruning process of $\alpha$-stable L\'evy trees \cite{AD2012} (i.e.\ when the pruning measure is a mixture of its length measure and a weighted measure on its branching points) is also the scaling limit of the dynamic mixture of bond and site percolation on $\alpha$-stable GW-trees \cite{AldousPitman1998b, ADHe2012}. 

This question posed in \cite[Example 4.6]{LohrVoisinWinter2015} remained unresolved. This work's primary contribution, Theorem \ref{NewTheo} presented below, provides a positive answer to this question. Although the result itself could be anticipated, the significance of our work is primarily technical; see further discussion at the end of Section \ref{Sec4}. 

We should note that He and Winkel \cite{HeW2014} (unfortunately unpublished) addressed this problem in a different topology, providing an answer even for general GW-trees and L\'evy trees. To be more precise, He and Winkel \cite{HeW2014} established the convergence in the Skorokhod space of c\`adl\`ag rooted locally compact $\mathbb{R}$-tree-valued functions equipped with the (localised) Gromov-Hausdorff topology. For a further discussion and the differences between the two topologies, we refer to \cite[Section 2.3]{HeW2014}.

In the next section, we define the space of bi-measure $\mathbb{R}$-trees and the LWV-topology. In Section~\ref{Sec3}, we define THE pruning process. This will enable us to formally state our main result in Section~\ref{Sec4}.

\subsection{Bi-measure $\mathbb{R}$-trees and the LWV-topology} \label{Sec2}

A rooted metric measure space is a quadruple $(\mathcal{X}, r, \rho, \mu)$,  where $(\mathcal{X}, r)$ is a metric space such that $(\text{supp}(\mu), r)$ is complete and separable, the so-called sampling measure $\mu$ is a finite measure on the Borel $\sigma$-algebra of $(\mathcal{X}, r)$ and $\rho \in \mathcal{X}$ is a distinguished point which is referred to as the root; the support $\text{supp}(\mu)$ of $\mu$ is defined as the smallest closet set $\mathcal{X}_{0} \subseteq \mathcal{X}$ such that $\mu(\mathcal{X}_{0}) = \mu(\mathcal{X})$. Recall that a rooted metric measure space is a ``tree-like'' metric space, or formally a $0$-hyperbolic space, iff satisfies the so-called $4$-point condition. Recall also that a $0$-hyperbolic space is called an $\mathbb{R}$-tree if it is connected; see \cite{And1996} for equivalent definitions and background on $\mathbb{R}$-trees.

For a (rooted measure) $\mathbb{R}$-tree $\mathcal{T} = (\mathcal{T},r,\rho, \mu)$, we denote the unique path between two points $u,v \in \mathcal{T}$ by $[u,v]$, and $[u,v[ \coloneqq [u,v] \setminus \{v\}$. For $n \in \mathbb{N}$ and $\mathbf{v}_{n} = (v_{1}, \dots, v_{n}) \in \mathcal{T}^{n}$, we use the notation $\llbracket \mathbf{v}_{n} \rrbracket$ for the tree spanned by the root $\rho$ and the vector $\mathbf{v}_{n}$, i.e., $\llbracket \mathbf{v}_{n} \rrbracket \coloneqq \bigcup_{i=1}^{n} [\rho, v_{i}]$. We write
\begin{equation} \label{e:117}
   \mathbf{Sk}_{\mu}(\mathcal{T}) \coloneqq \bigcup_{v \in \text{supp}(\mu)} [\rho, v[ \, \cup \, \{v \in \mathcal{T}: \mu(\{v \}) > 0 \},
\end{equation}
for the $\mu$-skeleton and
\begin{equation} \label{e:118}
  \mathbf{Lf}_{\mu}(\mathcal{T})  \coloneqq \llbracket \text{supp} (\mu) \rrbracket \setminus \mathbf{Sk}_{\mu}(\mathcal{T}),
\end{equation}
the set of $\mu$-leaves of $\mathcal{T}$.

\begin{definition}[{{\bf Bi-measure $\mathbb{R}$-trees} \cite[Section~2.3]{LohrVoisinWinter2015}}] \label{Def1}
We call $\mathcal{T} = (\mathcal{T}, r, \rho, \mu, \nu)$ a (rooted) bi-measure $\mathbb{R}$-tree if $(\mathcal{T}, r, \rho, \mu)$ is a $\mathbb{R}$-tree and $\nu$ is a $\sigma$-finite measure on $(\mathcal{T}, r)$ such that $\nu([\rho, u])$ is $\mu$-almost everywhere finite for $u \in \mathcal{T}$, and $\nu$ vanishes on the set of $\mu$-leaves, i.e., $\nu( \mathbf{Lf}_{\mu}(\mathcal{T})) = 0$.

Two bi-measure $\mathbb{R}$-trees $(\mathcal{T}, r, \rho, \mu, \nu)$ and $(\mathcal{T}^{\prime}, r^{\prime}, \rho^{\prime}, \mu^{\prime}, \nu^{\prime})$ are called equivalent if there exists a root preserving isometry $\phi: \llbracket {\rm supp}(\mu) \rrbracket  \rightarrow \mathcal{T}^{\prime}$ such that $\mu$ and $\nu$ are preserving in the $\mu$-skeleton. We denote by $\mathbb{H}_{f,  \sigma}$ the space of equivalence classes of (rooted) bi-measure $\mathbb{R}$-trees.
\end{definition}

In many interesting cases, the additional measure $\nu$ is not locally finite. For example, the length measure $\lambda_{\mathcal{T}}$ on an $\mathbb{R}$-tree $(\mathcal{T}, r, \rho, \mu)$ given by
\begin{eqnarray}
\lambda_{\mathcal{T}}([u,v]) = r(u,v), \hspace*{2mm} \text{for} \hspace*{2mm} u,v \in \mathcal{T}, \hspace*{2mm} \text{and} \hspace*{2mm} \lambda_{\mathcal{T}}(\mathbf{Lf}(T)) = 0,
\end{eqnarray}

\noindent where $\mathbf{Lf}(\mathcal{T}) \coloneqq \mathcal{T} \setminus \bigcup_{v \in \mathcal{T}} [\rho, v[$ denotes the set of leaves of $\mathcal{T}$, satisfies the properties of the measure $\nu$ in Definition \ref{Def1}. In particular, it is well-known that the length measure $\lambda_{\alpha}$ of the $\alpha$-stable L\'evy tree $\mathcal{T}_{\alpha}$ is not locally finite.

For $n \in \mathbb{N}$, a bi-measure $\mathbb{R}$-trees $(\mathcal{T}, r, \rho, \mu, \nu)$ and a (finite) sequence of points $\mathbf{u}_{n} = (u_{1}, \dots, u_{n}) \in \mathcal{T}^{n}$, we let $(\llbracket \mathbf{u}_{n}  \rrbracket, r, \rho, \mathbf{u}_{n}, \nu)$ be the $n$-pointed $\mathbb{R}$-tree spanned by the root $\rho$ and the points $\mathbf{u}_{n}$ equipped with $r$ and $\nu$, which we tacitly understand to be restricted to the appropriate space, i.e., $\llbracket \mathbf{u}_{n}  \rrbracket$.

\begin{definition}[{\bf LWV-topology}]  \label{Def2}
A sequence $(\mathcal{T}_{N})_{N \in \mathbb{N}}$ of bi-measure $\mathbb{R}$-trees converges to a bi-measure $\mathbb{R}$-tree $\mathcal{T}$ in $\mathbb{H}_{f, \sigma}$ in the leaf-sampling weak vague topology (LWV-topology) if $\mu(\mathcal{T}_{N}) \rightarrow \mu(\mathcal{T})$, and for any $n \in \mathbb{N}$,
\begin{equation}
(\llbracket \mathbf{u}_{n,N}  \rrbracket, r_{N}, \rho_{N}, \mathbf{u}_{n,N}, \nu_{N}) \rightarrow (\llbracket \mathbf{u}_{n}  \rrbracket, r, \rho, \mathbf{u}_{n}, \nu), \quad \text{in distribution}, \quad \text{as} \quad N \rightarrow \infty,
\end{equation}
\noindent with respect to the ($n$-pointed) Gromov-weak topology, where $\mathbf{u}_{n,N} = (u_{1,N}, \dots, u_{n,N})$ is a sequence of $n$ i.i.d.\ points with common distribution $\mu_{N}(\cdot)/\mu_{N}(\mathcal{T}_{N})$. Similarly, $\mathbf{u}_{n} = (u_{1}, \dots, u_{n})$ is a sequence of $n$ i.i.d.\ points with common distribution $\mu(\cdot)/\mu(\mathcal{T})$.
\end{definition}

Note that our definition of the LWV-topology is equivalent to that in \cite[Definition~2.14]{LohrVoisinWinter2015} due to \cite[Proposition~2.17]{LohrVoisinWinter2015}. This alternative characterization of convergence will be useful in our proofs. The space $\mathbb{H}_{f,\sigma}$ equipped with the LWV-topology is separable and metrizable; see \cite[Corollary~2.19]{LohrVoisinWinter2015}. For background and the definition of the ($n$-pointed) Gromov-weak topology, we refer to Section \ref{Sec5}. 

\subsection{The pruning process}  \label{Sec3}

Consider a bi-measure $\mathbb{R}$-tree $\mathcal{T} =(\mathcal{T}, r, \rho, \mu, \nu) \in \mathbb{H}_{f,  \sigma}$ and a subset $\pi \subseteq \mathbb{R}_{+} \times \mathcal{T}$. Although $\pi$ is associated with a particular class representative, it corresponds, of course, to a similar set for any representative of the same equivalence class by mapping across using the appropriate root invariant isometry. Define the set of cut points up to time $t \geq 0$ by $\pi_{t} \coloneqq \{v \in \mathcal{T}: \exists \, \,  s \leq t \, \, \text{such that} \, \, (s,v) \in \pi \}$. For every $v \in \mathcal{T}$, the tree pruned at $v$ is defined by
\begin{equation*}
\mathcal{T}^{v} \coloneqq \{w \in \mathcal{T}: v \notin [\rho, w] \}.
\end{equation*}
Then the pruned tree at the set $\pi_{t} \subseteq \mathcal{T}$, it is the intersection of the trees $T^{v}$ pruned at $v \in \pi_{t}$,
\begin{equation*}
\mathcal{T}^{\pi_{t}} \coloneqq \bigcap_{v \in \pi_{t}} \mathcal{T}^{v}.
\end{equation*}
We naturally equip the pruned tree $\mathcal{T}^{\pi_{t}}$ with the restrictions of the metric $r$ and the measures $\mu$ and $\nu$, i.e.\ we write $\mathcal{T}^{\pi_{t}} = (\mathcal{T}^{\pi_{t}}, r, \rho, \mu, \nu) = (\mathcal{T}^{\pi_{t}}, r \upharpoonright_{\mathcal{T}^{\pi_{t}} \times \mathcal{T}^{\pi_{t}}}, \rho, \mu \upharpoonright_{\mathcal{T}^{\pi_{t}}}, \nu \upharpoonright_{\mathcal{T}^{\pi_{t}}})$ that is an element of $\mathbb{H}_{f, \sigma}$.

\begin{definition}[{{\bf THE pruning process} \cite[Definition~3.2]{LohrVoisinWinter2015}}]
Fix a bi-measure $\mathbb{R}$-tree $\mathcal{T} = (\mathcal{T}, r, \rho, \mu, \nu) \in \mathbb{H}^{f,  \sigma}$. Let $\pi=\pi^{\mathcal{T}}$ be a Poisson point measure on $\mathbb{R}_{+} \times \mathcal{T}$ with intensity measure ${\rm d} t \otimes \nu$, where ${\rm d} t $ is the Lebesgue measure on $\mathbb{R}_{+}$. We define the pruning process, $Y^{\mathcal{T}} \coloneqq (Y^{\mathcal{T}}(t), t \geq 0)$, as the bi-measure $\mathbb{R}$-tree-valued process obtained by pruning $Y^{\mathcal{T}}(0) \coloneqq (\mathcal{T}, r, \rho, \mu, \nu)$ at the points of the Poisson point process $\pi_{t}(\cdot) = \pi_{t}^{\mathcal{T}}(\cdot) \coloneqq \pi^{\mathcal{T}}([0,t] \times \cdot)$, i.e.,
\begin{equation*}
Y^{\mathcal{T}}(t) \coloneqq (\mathcal{T}^{\pi_{t}}, r, \rho, \mu, \nu).
\end{equation*}
We call $\nu$ the pruning measure of $Y^{\mathcal{T}}$.
 \label{Def3}
\end{definition}

Informally, since the pruning measure is already part of the state space, the pruning process can be described as a pure jump process which, given a bi-measure $\mathbb{R}$-tree $(\mathcal{T}, d, \rho, \mu, \nu)$, lets rain down successively more and more cut points according to a Poisson process whose intensity measure is equal to the pruning measure. At each cut point, the subtree above is cut off and removed, and the sampling measure $\mu$ and pruning measure $\nu$ are simultaneously updated by simply restricting them to the remaining, pruned part of the $\mathbb{R}$-tree. By  \cite[Lemma 3.1, Lemma 3.3 and Proposition 3.4]{LohrVoisinWinter2015}, the pruning process is a Feller-continuous strong Markov $\mathbb{H}_{f,  \sigma}$-valued process  (equipped with the LWV-topology) with c\`adl\`ag paths.

\subsection{Main results}  \label{Sec4}

In this section, we present our main result. Specifically, we establish the convergence of the pruning process of $\alpha$-stable ${\rm GW}$-trees (i.e.\ dynamic percolation) to the pruning process of $\alpha$-stable L\'evy trees. We begin by posing the problem in the context of $\mathbb{R}$-trees and bi-measure $\mathbb{R}$-trees.

Recall that an $\alpha$-stable ${\rm GW}$-tree is a critical GW-tree whose offspring distributions lie in the domain of attraction of an $\alpha$-stable law, conditioned to have a fixed total progeny $N \in \mathbb{N}$. Specifically,
the offspring distribution $\eta = (\eta(k), k \geq 0)$ is a probability distribution on $\{0, 1, \dots \}$ satisfying $\sum_{k = 0}^{\infty} k \eta(k) = 1$. We always implicitly assume that $\eta(0) >0$, $\eta(0) + \eta(1) < 1$ to avoid degenerate cases, and that $\eta$ is aperiodic (we introduce this last condition to avoid unnecessary complications). We say that  $\eta $ is in the domain of attraction of a stable law of index $\alpha \in (1,2]$ if either the variance of $\eta$ is finite, or $\eta ([k, \infty)) = k^{-\alpha} L(k)$, where $L: \mathbb{R}_{+} \rightarrow \mathbb{R}_{+}$ is a function such that $L(x) >0$ for $x$ large enough and $\lim_{x \rightarrow \infty} L(tx)/L(x) = 1$ for all $t >0$ (such function is called slowly varying function).

Let $(\mathbf{t}_{N}, r_{N}^{\text{gr}}, \rho_{N}, \mu_{N}^{\rm nod})$ be an $\alpha$-stable ${\rm GW}$-tree, viewed as a (rooted measure) $\mathbb{R}$-tree with unit length edges (equipped with the graph-distance and the uniform measure on its set of vertices). Let $\ell: \mathbb{R}_{+} \rightarrow \mathbb{R}_{+}$ be a slowly varying function for which the convergence in (\ref{eq1}) holds. We define two sequences of positive real numbers $(a_{N})_{N \in \mathbb{N}}$ and $(b_{N})_{N \in \mathbb{N}}$ by letting
\begin{eqnarray} \label{eq23}
a_{N} \coloneqq \ell(N)N^{\frac{1}{\alpha} -1} \hspace*{3mm} \text{and} \hspace*{3mm} b_{N} \coloneqq \ell(N)^{-1}N^{-\frac{1}{\alpha}}, \hspace*{4mm} \text{for} \hspace*{2mm} N \in \mathbb{N}.
\end{eqnarray}

\noindent As in the introduction, consider the rescaled $\alpha$-stable ${\rm GW}$-tree $a_{N} \cdot \mathbf{t}_{N} = (\mathbf{t}_{N}, a_{N} \cdot r_{N}^{\text{gr}}, \rho_{N}, \mu_{N}^{\rm nod})$.  Next, we introduce two different pruning measures on $a_{N} \cdot \mathbf{t}_{N}$, 
\begin{eqnarray} \label{eq7}
\nu^{\text{ske}}_{N} \coloneqq a_{N} \cdot \lambda_{N}, \hspace*{3mm} \text{and} \hspace*{3mm} \nu_{N}^{\rm bra} (v) \coloneqq b_{N} \cdot (c(v) -1) \mathbf{1}_{\{ c(v) \geq 1 \}}, \hspace*{4mm} \text{for} \hspace*{2mm} v \in \mathbf{t}_{N},
\end{eqnarray}

\noindent where $\lambda_{N}$ is the length measure on $(\mathbf{t}_{N}, r_{N}^{\text{gr}}, \rho_{N}, \mu_{N}^{\rm nod})$ and $c(v)$ denotes the number of children of a vertex $v$ in $\mathbf{t}_{N}$. We recall that a vertex $v \in \mathbf{t}_{N}$ with $c(v) \geq 1$ is called a branching-point. Note also that $\nu^{\text{ske}}_{N}$ is indeed the length measure on $a_{N} \cdot \mathbf{t}_{N}$. The first pruning measure corresponds to cutting the edges (i.e.\ bond-percolation) as in \cite{AldousPitman1998b}, while the second one corresponds to cutting the vertices (i.e.\ site-percolation) as in \cite{ADHe2012}. We also consider a mixture of the previous pruning measures, i.e.\ $\nu_{N}^{\rm mix} \coloneqq \nu^{\text{ske}}_{N} + \nu_{N}^{\rm bra}$.

Let $\mathcal{T}_{\alpha} = (\mathcal{T}_{\alpha}, r_{\alpha}, \rho_{\alpha}, \mu_{\alpha})$ be the $\alpha$-stable L\'evy tree of index $\alpha \in (1,2]$. Recall that $\mathcal{T}_{\alpha}$ is a compact $\mathbb{R}$-tree and $\mu_{\alpha}$ is the uniform probability measure supported on its set of leaves; see \cite{Du2005}. Recall also that the multiplicity of a point $x \in \mathcal{T}_{\alpha}$ is the number of connected components of $\mathcal{T}_{\alpha} \setminus \{x\}$ and that the degree of a point is defined as its multiplicity minus one; points with degree $0$ are called leaves, and those with degree at least $2$ are called branching-points. Indeed, by \cite[Theorem 4.6]{Du2005}, the degree of every point in $\mathcal{T}_{\alpha}$ belongs to $\{0,2, \infty \}$. In particular, for $\alpha = 2$ (i.e.\ the Brownian CRT case), the branching-points have exactly degree $2$. For $\alpha \in (1,2)$, Duquesne and Le Gall \cite[Theorem 4.6]{Du2005} have shown that a.s.\ the branching-points of $\mathcal{T}_{\alpha}$ have infinite degree and that they form a countable set, say $\mathcal{B}_{\alpha}$. Moreover, for any $x \in \mathcal{B}_{\alpha}$, one can define the local time, or width of $x$ as the almost sure limit
\begin{eqnarray} \label{eq18}
\Delta_{x} = \lim_{\varepsilon \rightarrow 0^{+}}  \frac{1}{\varepsilon} \, \, \mu_{\alpha} (\left \{ y \in \mathcal{T}_{\alpha}: \, \, x \in [\rho_{\alpha}, y],  \, \, r_{\alpha}(x,y) < \varepsilon \right \}).
\end{eqnarray}

\noindent The existence of this quantity is justified in \cite[Proposition 2]{Miermont2005}.

We then consider two different pruning measures in $\mathcal{T}_{\alpha}$. The length measure $\nu_{\alpha}^{\rm ske} \coloneqq \lambda_{\alpha}$ on $\mathcal{T}_{\alpha}$ which corresponds to pruning proportional to the length of its skeleton as in \cite{Aldous1998a, AS2002}. For $\alpha \in (1,2)$, the $\sigma$-finite measure,
\begin{eqnarray}
\nu^{{\rm bra}}_{\alpha} (\{x\}) \coloneqq \Delta_{x}, \hspace*{4mm}  \text{for} \hspace*{2mm} x \in \mathcal{B}_{\alpha},
\end{eqnarray}

\noindent which corresponds to pruning on the infinite branching points of $\mathcal{T}_{\alpha}$ as in \cite{Miermont2005, AD2008}; \cite[Lemma 4]{Miermont2005} guarantees that a.s.\ $\nu^{{\rm bra}}$ satisfies the properties of the measure $\nu$ in Definition \ref{Def1}. We also consider a mixture of the previous pruning measures, i.e., $\nu^{{\rm mix}}_{\alpha} \coloneqq  \nu_{\alpha}^{\rm ske} + \nu_{\alpha}^{\rm bra} \mathbf{1}_{\{\alpha \in (1,2)\}} + \nu_{2}^{\rm ske} \mathbf{1}_{\{\alpha =2\}}$, where $ \nu_{2}^{\rm ske}$ is the length measure on the Brownian CRT. This last pruning measure corresponds to pruning proportional to the length and on the infinite branching points of $\mathcal{T}_{\alpha}$ as in \cite{ADV2010, AD2012}.

We have now all the ingredients to state our main result. We write $\mathbb{D}([0,\infty), \mathbb{H}_{f, \sigma})$ for the space of $\mathbb{H}_{f, \sigma}$-valued c\`adl\`ag functions equipped with the Skorohod ${\rm J}_{1}$ topology. The space $\mathbb{H}_{f, \sigma}$ is equipped with the LWV-topology.

\begin{theorem} \label{NewTheo}
For $N \in \mathbb{N}$ and $\ast \in \{ {\rm ske}, {\rm bra}, {\rm mix}\}$, let $\mathbf{t}_{N}^{\ast} = (\mathbf{t}_{N}, a_{N} \cdot r_{N}^{{\rm gr}}, \rho_{N}, \mu_{N}^{\rm nod}, \nu_{N}^{\ast})$ be the bi-measure $\mathbb{R}$-tree associated with an $\alpha$-stable ${\rm GW}$-tree. Similarly, let $\mathcal{T}_{\alpha}^{\ast} = (\mathcal{T}_{\alpha}, r_{\alpha}, \rho_{\alpha}, \mu_{\alpha}, \nu^{\ast}_{\alpha})$ be the bi-measure $\mathbb{R}$-tree associated with an $\alpha$-stable L\'evy tree. For $\ast \in \{ {\rm ske}, {\rm bra}, {\rm mix}\}$, we have that, in distribution,
\begin{eqnarray}
Y^{\mathbf{t}_{N}^{\ast}} \xRightarrow[]{{\rm Sk}}  Y^{\mathcal{T}_{\alpha}^{\ast}}, \quad \text{as}  \quad N \rightarrow \infty, \quad \text{in the space} \quad \mathbb{D}([0,\infty), \mathbb{H}_{f, \sigma}). 
\end{eqnarray} 
\end{theorem}

Let $\mathcal{T} = (\mathcal{T}, r, \rho, \mu, \nu)$ be a bi-measure $\mathbb{R}$-tree and define the pruning process $Z^{\mathcal{T}} \coloneqq (Z^{\mathcal{T}}(t), t \geq 0)$, as a $\mathbb{R}$-tree-valued process, by letting $Z^{\mathcal{T}}(0) \coloneqq (\mathcal{T}, r, \rho, \mu)$ and
\begin{equation} 
 Z^{\mathcal{T}}(t) \coloneqq (\mathcal{T}^{\pi_{t}}, r, \rho, \mu) = (\mathcal{T}^{\pi_{t}}, r \upharpoonright_{\mathcal{T}^{\pi_{t}} \times \mathcal{T}^{\pi_{t}}}, \rho, \mu \upharpoonright_{\mathcal{T}^{\pi_{t}}}),
\end{equation}
\noindent for $t \geq 0$ (i.e.\ $Z^{\mathcal{T}}$ is the pruning process $Y^{\mathcal{T}}$ where we do not keep track of the pruning measure). Let $\mathbb{H}_{0}$ be the space of (rooted) rooted metric measure spaces (see Definition \ref{Def4} below) equipped with the Gromov-weak topology. The following corollary is a consequence of Theorem \ref{NewTheo}. We write $\mathbb{D}([0,\infty), \mathbb{H}_{0})$ for the space of $\mathbb{H}_{0}$-valued c\`adl\`ag functions equipped with the Skorohod ${\rm J}_{1}$ topology.

\begin{corollary} \label{Cor1}
For $\ast \in \{ {\rm ske}, {\rm bra}, {\rm mix}\}$, we have that, in distribution,
\begin{eqnarray}
Z^{\mathbf{t}_{N}^{\ast}} \xRightarrow[]{{\rm Sk}}  Z^{\mathcal{T}_{\alpha}^{\ast}}, \quad \text{as}  \quad N \rightarrow \infty, \quad \text{in the space} \quad \mathbb{D}([0,\infty), \mathbb{H}_{0}).
\end{eqnarray}
\end{corollary}

\begin{proof}
Since LWV-convergence implies Gromov-weak convergence (\cite[Remark 2.16]{LohrVoisinWinter2015}), our claim follows from Theorem \ref{NewTheo}.
\end{proof}

The pruning process $Z^{\mathbf{t}_{N}^{\rm ske}}$ is, up to the time transformation $u = e^{-ta_{N}}$, the same as the pruning process uniformly on the edges of Aldous and Pitman \cite{AldousPitman1998b}. The pruning process $Z^{\mathcal{T}_{\alpha}^{\rm ske}}$ is precisely the one considered by Aldous and Pitman \cite{Aldous1998a} and by Abraham and Serlet \cite{AS2002}. Note also that $Z^{\mathbf{t}_{N}^{\rm bra}}$ is, up to the time transformation $u = e^{-tb_{N}}$, the same as the pruning process in the branching points of an $\alpha$-stable ${\rm GW}$-tree in \cite{ADHe2012}. Moreover, $Z^{\mathcal{T}_{\alpha}^{\rm bra}}$ is the one considered by Miermont \cite{Miermont2005} and by Abraham and Delmas \cite{AD2008} in the stable case. The pruning process $Z^{\mathcal{T}_{\alpha}^{\rm mix}}$ corresponds to that in \cite{ADV2010, AD2012} for the $\alpha$-stable L\'evy tree. On the other hand,  L\"ohr, Voisin and Winter \cite[Section 4]{LohrVoisinWinter2015} already showed that the sequence $(Y^{\mathbf{t}_{N}^{\rm ske}})_{N \in \mathbb{N}}$ converges to $Y^{\mathcal{T}_{\alpha}^{\rm ske}}$. 

As mentioned in the introduction, Theorem \ref{NewTheo} positively answers the question posed in \cite{LohrVoisinWinter2015}. While the result itself could be anticipated, the primary contribution of our work lies in its technical approach. The key ingredients are Theorem \ref{Theo4} and Lemma \ref{lemma7}, which provide an estimation of the degree distribution along specific branches of the $\alpha$-stable GW-tree. Informally, this corresponds to the value of the pruning measure $\nu_{N}^{\rm bra}$ for finite intervals. The proofs of both Theorem \ref{Theo4} and Lemma \ref{lemma7} are quite technical, with the main difficulties relying on establishing continuity properties of certain functionals in the Skorokhod ${\rm J}_{1}$ topology (see Appendix \ref{Append1}). These properties might be of independent interest. Furthermore, the proof of Theorem \ref{NewTheo} introduces a new approach to establishing convergence between bi-measure $\mathbb{R}$-trees with respect to the LWV-topology using coding functions, a method that differs from the applications found in \cite{LohrVoisinWinter2015}.

The remainder of the article is organised as follows. In Section \ref{Sec5}, we recall the Gromov-weak topology on the space of rooted metric measure spaces. In Section~\ref{Sec10}, we recall some relevant properties of Galton-Watson trees and $\alpha$-stable L\'evy trees. We also prove Theorem~\ref{Theo4} and Lemma \ref{lemma7}, results that may be of independent interest and are used in the proof of Theorem~\ref{NewTheo} in Section \ref{Sec15}. Readers familiar with the framework can skip Sections \ref{Sec5} and \ref{Sec10} and proceed directly to the proof of Theorem~\ref{NewTheo}.

\section{The Gromov-weak topology} \label{Sec5}

We start recalling basic notation. For a topological space $\mathcal{X}$ and a finite measure $\mu$ defined on the Borel $\sigma$-algebra of $\mathcal{X}$, the push forward of $\mu$ under a measurable map $\phi$ from $\mathcal{X}$ into another topological space $\mathcal{Z}$ is the finite measure $\phi_{\ast} \mu$ defined by $\phi_{\ast} \mu(A) \coloneqq \mu (\phi^{-1}(A))$, for all measurable sets $A \subseteq \mathcal{Z}$.

A (rooted) $n$-pointed metric measure space $(\mathcal{X}, r, \rho, \mathbf{u}_{n}, \mu)$ consists of a rooted metric measure space $(\mathcal{X}, r, \rho, \mu)$ and $n \in \mathbb{N}$ additional fixed points $\mathbf{u}_{n} = (u_{1}, \dots, u_{n}) \in \mathcal{X}^{n}$. The support of an $n$-pointed metric measure space $\mathcal{X} = (\mathcal{X}, r, \rho, \mathbf{u}_{n}, \mu)$ is defined by $\text{supp}(\mathcal{X}) \coloneqq \text{supp}(\mu) \cup \{ \rho, u_{1}, \dots, u_{n} \}$; we refer to \cite{Grev2009, LohrVoisinWinter2015} for more details and discussions on metric measure spaces. In particular, for $n=0$, the (rooted) $0$-pointed metric measure space is only the usual rooted metric measure space $(\mathcal{X}, r, \rho, \mu)$.

\begin{definition} \label{Def4}
Two (rooted) $n$-pointed metric measure spaces $(\mathcal{X}, r, \rho, \mathbf{u}_{n}, \mu)$ and $(\mathcal{X}^{\prime}, r^{\prime}, \rho^{\prime}, \mathbf{u}_{n}^{\prime}, \mu^{\prime})$ are called equivalent if there exists an isometry $\phi: {\rm supp}(\mathcal{X})  \rightarrow {\rm supp}(\mathcal{X}^{\prime})$ such that $\phi_{\ast} \mu = \mu^{\prime}$, $\phi(\rho) = \rho^{\prime}$ and $\phi(u_{k}) = u_{k}^{\prime}$, for all $1 \leq k \leq n$. We denote by $\mathbb{H}_{n}$ the space of (rooted) $n$-pointed metric measure spaces. In particular, $\mathbb{H}_{0}$ denotes the usual space of rooted metric measure spaces.
\end{definition}

For a rigorous definition of the Gromov-weak topology, we refer to Gromov's book \cite{Grom2007} or the works \cite{Grev2009, LohrVoisinWinter2015}. Here, we provide a brief overview: the Gromov-weak topology is metrized by the so-called pointed Gromov-Prokhorov metric $d_{\rm pGP}$; see \cite[Proposition 2.6]{LohrVoisinWinter2015}. Let  $d_{\rm Pr}$ be the Prohorov distance between two finite measures $\mu, \mu^{\prime}$ on a metric space $(E, r_{E})$, i.e.,
\begin{eqnarray} \label{Proh1}
d_{\rm Pr}^{(E, r_{E})}(\mu, \mu^{\prime}) \coloneqq \inf \left \{ \varepsilon > 0: \mu(A^{\varepsilon}) + \varepsilon \geq \mu^{\prime}(A), \hspace*{2mm} \mu^{\prime}(A^{\varepsilon}) + \varepsilon \geq \mu(A)  \hspace*{2mm} \text{for all} \hspace*{2mm} A  \subseteq E \hspace*{2mm} \text{closed} \right\},
\end{eqnarray}

\noindent where $A^{\varepsilon} := \{ x \in E: r_{E}(x,A) < \varepsilon \}$. The pointed Gromov-Prohorov distance between two $n$-pointed metric measure spaces $(\mathcal{X}, r, \rho, \mathbf{u}_{n}, \mu)$ and $(\mathcal{X}^{\prime}, r^{\prime}, \rho^{\prime}, \mathbf{u}_{n}^{\prime}, \mu^{\prime})$ in $\mathbb{H}_{n}$ is given by
\begin{eqnarray} \label{Proh2}
d_{\rm pGP}(\mathcal{X}, \mathcal{X}^{\prime}) \coloneqq \inf_{d} \left \{ d_{\rm Pr}^{(\mathcal{X} \uplus \mathcal{X}^{\prime}, d)}(\mu, \mu^{\prime}) + d(\rho, \rho^{\prime})  + \sum_{k=1}^{n} d(u_{k}, u_{k}^{\prime}) \right\},
\end{eqnarray}

\noindent where the infimum is taken over all metrics $d$ on the disjoint union $\mathcal{X} \uplus \mathcal{X}^{\prime}$ that extends $r$ and $r^{\prime}$. In particular, $(\mathbb{H}_{n}, d_{\rm pGP})$ is a complete and separable metric space; see \cite[Proposition 2.6]{LohrVoisinWinter2015}.

\section{Galton-Watson trees and random walks}  \label{Sec10}

\subsection{Definitions}  \label{Sec11}

We define rooted plane trees following \cite{Ne1986, Leje2005}. Let $\mathbb{N} = \{1, 2, \dots \}$ be the set of positive integers, set $\mathbb{N}^{0} = \{ \varnothing  \}$ and consider the set of labels $\mathbb{U} = \bigcup_{n \geq 0} \mathbb{N}^{n}$. For $u = (u_{1}, \dots, u_{n}) \in \mathbb{U}$, we denote by $|u| = n $ the length (or generation, or height) of $u$; if $v = (v_{1}, \dots, v_{m}) \in \mathbb{U}$, we let $uv = (u_{1}, \dots, u_{n}, v_{1}, \dots, v_{m}) \in \mathbb{U}$ be the concatenation of $u$ and $v$. A rooted plane tree is a nonempty, finite subset $\tau \subset  \mathbb{U}$ such that: (i) $\varnothing \in \tau$; (ii) if $v \in \tau$ and $v = uj$ for some $j \in \mathbb{N}$, then $u \in \tau$; (iii) if $u \in \tau$, then there exists an integer $c(u) \geq 0$ such that $ui \in  \tau$ if and only if $1 \leq i \leq c(u)$. We view each vertex $u$ in a plane tree $\tau$ as an individual of a population with genealogical tree $\tau$. The vertex $\varnothing$ is called the root of the tree and for every $u \in \tau$, $c(u)$ is the number of children of $u$ (if $c(u) = 0$, then $u$ is called a leaf, otherwise, $u$ is called an internal vertex). The total progeny of $\tau$ will be denoted by $\zeta(\tau) = \textsf{Card}(\tau)$ (i.e., the number of vertices of $\tau$). If $u,v \in \tau$ we denote by $[u,v]$ the discrete geodesic path (of vertices) between $u$ and $v$ in $\tau$. We also write $[u,v[ \coloneqq [u,v] \setminus \{v\}$. We denote by $\mathbb{T}$ the set of finite plane trees.

We write $u \preccurlyeq v$ if $v = uw$ for some $w \in \mathbb{U}$ ($\preccurlyeq$ is the ``genealogical'' order). In particular, we write $u \prec v$ if $u \neq v$. We also denote by $u \wedge v$ the most recent common ancestor of $u$ and $v$, i.e., $u \wedge v  \coloneqq \sup \{ w \in \tau: w \preccurlyeq u \, \, \text{and} \, \, w \preccurlyeq v\}$, where the supremum is taken for the genealogical order.

\subsection{Coding planar trees by discrete paths}  \label{Sec12}

Let $\tau \in \mathbb{T}$ be a rooted plane tree $\tau \in \mathbb{T}$. In this work, we will use two different orderings of the vertices of a tree $\tau$:
\begin{enumerate}[label=(\textbf{\roman*})]
\item \textbf{Lexicographical ordering.} Given $u,w \in \tau$, we write $v \prec_{\uparrow} w$ if there exits $z \in \tau$ such that $v = z(v_{1}, \dots, v_{n})$, $w = z(w_{1}, \dots, w_{m})$ and $v_{1} < w_{1}$.
\item \textbf{Reverse-lexicographical ordering.} Given $u,w \in \tau$, we write $v \prec_{\downarrow} w$ if there exists $z \in \tau$ such that $v = z (v_{1},...,v_{n})$, $w =z(w_{1},...,w_{m})$ and $v_{1} > w_{1}$.
\end{enumerate}

\begin{figure}[!htb]
   \begin{minipage}{0.48\textwidth}
     \centering
      \begin{tikzpicture}[scale=.6,font=\footnotesize]
\tikzstyle{solid node}=[circle,draw,inner sep=1.5,fill=black]
    \tikzstyle{level 1}=[sibling distance=35mm]
    \tikzstyle{level 2}=[sibling distance=15mm]
    \tikzstyle{level 3}=[sibling distance=10mm]
\node[solid node, label=below:{$0$}]{}
[grow=north]
child{node[solid node, label=left:{$9$}]{}
   child{node[solid node, label=left:{$16$}]{}}
   child{node[solid node, label=left:{$15$}]{}}
   child{node[solid node, label=left:{$11$}]{}
     child{node[solid node, label=left:{$12$}]{}
        child{node[solid node, label=left:{$14$}]{}}
        child{node[solid node, label=left:{$13$}]{}}
       }
     }
   child{node[solid node, label=left:{$10$}]{}}
}
child{node[solid node, label=left:{$7$}]{}
   child{node[solid node, label=left:{$8$}]{}}
}
child{node[solid node, label=left:{$1$}]{}
   child{node[solid node, label=left:{$3$}]{}
     child{node[solid node, label=left:{$6$}]{}}
     child{node[solid node, label=left:{$5$}]{}}
     child{node[solid node, label=left:{$4$}]{}}
   }
   child{node[solid node, label=left:{$2$}]{}}
}
;
\end{tikzpicture}
   \end{minipage}\hfill
   \begin {minipage}{0.48\textwidth}
     \centering
    \begin{tikzpicture}[scale=.6,font=\footnotesize]
\tikzstyle{solid node}=[circle,draw,inner sep=1.5,fill=black]
    \tikzstyle{level 1}=[sibling distance=35mm]
    \tikzstyle{level 2}=[sibling distance=15mm]
    \tikzstyle{level 3}=[sibling distance=10mm]
\node[solid node, label=below:{$0$}]{}
[grow=north]
child{node[solid node, label=left:{$1$}]{}
   child{node[solid node, label=left:{$2$}]{}}
   child{node[solid node, label=left:{$3$}]{}}
   child{node[solid node, label=left:{$4$}]{}
     child{node[solid node, label=left:{$5$}]{}
        child{node[solid node, label=left:{$6$}]{}}
        child{node[solid node, label=left:{$7$}]{}}
       }
     }
   child{node[solid node, label=left:{$8$}]{}}
}
child{node[solid node, label=left:{$9$}]{}
   child{node[solid node, label=left:{$10$}]{}}
}
child{node[solid node, label=left:{$11$}]{}
   child{node[solid node, label=left:{$12$}]{}
     child{node[solid node, label=left:{$13$}]{}}
     child{node[solid node, label=left:{$14$}]{}}
     child{node[solid node, label=left:{$15$}]{}}
   }
   child{node[solid node, label=left:{$16$}]{}}
}
;
\end{tikzpicture}
   \end{minipage}
\caption{From left to right, a plane tree with vertices labeled in lexicographical and reverse lexicographical order.}\label{Fig2}
\end{figure}
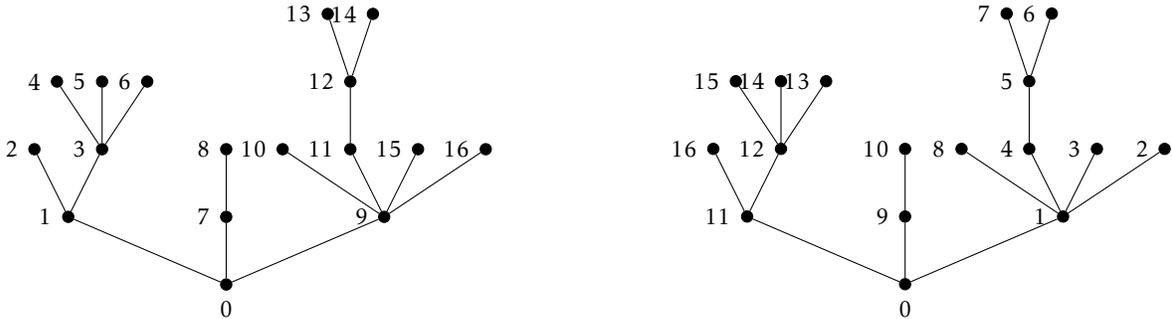

\noindent For $\ast \in \{\uparrow, \downarrow \}$, we associate to every ordering $\varnothing = u(0) \prec_{\ast} u(1) \prec_{\ast} \dots \prec_{\ast} u(\zeta(\tau)-1)$ of the vertices of $\tau$ a path $W^{\ast} = (W^{\ast}(n), 0 \leq n \leq \zeta(\tau))$, by letting $W^{\ast}(0) = 0$, and for $0 \leq n \leq \zeta(\tau)-1$, $W^{\ast}(n+1) = W^{\ast}(n)  + c(u(n))-1$. The path $W^{\uparrow}$ is commonly called the $\L$ukasiewicz path of $\tau$, and through this work, we will call $W^{\downarrow}$ the reverse-$\L$ukasiewicz path. It is easy to see that, for every $\ast \in \{\uparrow, \downarrow \}$, $W^{\ast}(n) \geq 0$ for every $0 \leq n \leq \zeta(\tau)-1$ but $W^{\ast}(\zeta(\tau)) = -1$. We refer to \cite{Leje2005} for details, properties and proofs of the assertions we will use on the $\L$ukasiewicz path (clearly, similar ones hold for the reverse-$\L$ukasiewicz path).  By definition,
\begin{eqnarray}
\Delta W^{\ast}(n+1) \coloneqq W^{\ast}(n+1) - W^{\ast}(n) = c(u(n)) -1 \geq -1, \hspace*{5mm} \text{for} \hspace*{5mm}
0 \leq n \leq \zeta(\tau)-1,
\end{eqnarray}

\noindent with equality if and only if $u(n)$ is a leaf of $\tau$. We shall think of such paths as step functions on $[0,\zeta(\tau)]$ given by $s \mapsto W^{\ast}(\lfloor s \rfloor)$, for $\ast \in \{\uparrow, \downarrow \}$; see Figure \ref{Fig3}.

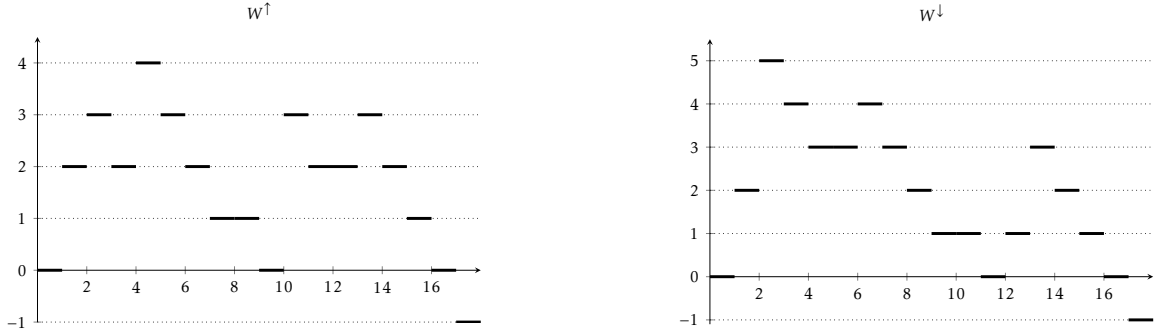
\begin{figure}[!htb]
   \begin{minipage}{0.48\textwidth}
     \centering
\begin{tikzpicture}[scale=.7,font=\footnotesize]
\begin{axis}[axis x line = middle,axis y line = left, title = $W^{\uparrow}$, width=10cm, height=7cm, xtick={0,2,...,16}, ytick={-1,0,...,4}, xmin=0, xmax=18,  ymin=-1, ymax=4.5]
\addplot[domain=0:1, ultra thick] {0};
\addplot[domain=1:2, ultra thick] {2};
\addplot[domain=2:3, ultra thick] {3};
\addplot[domain=3:4, ultra thick] {2};
\addplot[domain=4:5, ultra thick] {4};
\addplot[domain=5:6, ultra thick] {3};
\addplot[domain=6:7, ultra thick] {2};
\addplot[domain=7:8, ultra thick] {1};
\addplot[domain=8:9, ultra thick] {1};
\addplot[domain=9:10, ultra thick] {0};
\addplot[domain=10:11, ultra thick] {3};
\addplot[domain=11:12, ultra thick] {2};
\addplot[domain=12:13, ultra thick] {2};
\addplot[domain=13:14, ultra thick] {3};
\addplot[domain=14:15, ultra thick] {2};
\addplot[domain=15:16, ultra thick] {1};
\addplot[domain=16:17, ultra thick] {0};
\addplot[domain=17:18, ultra thick] {-1};
\addplot[domain=0:18, dotted] {-1};
\addplot[domain=0:18, dotted] {1};
\addplot[domain=0:18, dotted] {2};
\addplot[domain=0:18, dotted] {3};
\addplot[domain=0:18, dotted] {4};
 \end{axis}
\end{tikzpicture}
   \end{minipage}\hfill
   \begin {minipage}{0.48\textwidth}
     \centering
    \begin{tikzpicture}[scale=.7,font=\footnotesize]
\begin{axis}[axis x line = middle,axis y line = left, title = $W^{\downarrow}$,  width=10cm, height=7cm, xtick={0,2,...,16}, ytick={-1,0,1,...,5}, xmin=0, xmax=18,  ymin=-1.1, ymax=5.5]
\addplot[domain=0:1, ultra thick] {0};
\addplot[domain=1:2, ultra thick] {2};
\addplot[domain=2:3, ultra thick] {5};
\addplot[domain=3:4, ultra thick] {4};
\addplot[domain=4:5, ultra thick] {3};
\addplot[domain=5:6, ultra thick] {3};
\addplot[domain=6:7, ultra thick] {4};
\addplot[domain=7:8, ultra thick] {3};
\addplot[domain=8:9, ultra thick] {2};
\addplot[domain=9:10, ultra thick] {1};
\addplot[domain=10:11, ultra thick] {1};
\addplot[domain=11:12, ultra thick] {0};
\addplot[domain=12:13, ultra thick] {1};
\addplot[domain=13:14, ultra thick] {3};
\addplot[domain=14:15, ultra thick] {2};
\addplot[domain=15:16, ultra thick] {1};
\addplot[domain=16:17, ultra thick] {0};
\addplot[domain=17:18, ultra thick] {-1};
\addplot[domain=0:18, dotted] {-1};
\addplot[domain=0:18, dotted] {1};
\addplot[domain=0:18, dotted] {2};
\addplot[domain=0:18, dotted] {3};
\addplot[domain=0:18, dotted] {4};
\addplot[domain=0:18, dotted] {5};
 \end{axis}
\end{tikzpicture}
   \end{minipage}
        \caption{From left to right, the paths  $W^{\uparrow}$ and $W^{\downarrow}$ of the labeled plane trees in Figure \ref{Fig2}.} \label{Fig3}
\end{figure}

Note the following useful elementary properties of these coding paths:
\begin{enumerate}[label=(\textbf{P.\arabic*})]
\item For every $0 \leq  n \leq \zeta(\tau) - 1$, $W^{\uparrow}(n)$ is equal to the number of children of vertices in the path $[ u(0), u(n)[$  branching
to the right of $[ u(0), u(n)[$. \label{P1}

\item For every $0 \leq  n\leq \zeta(\tau) - 1$, $W^{\downarrow}(n)$ is equal to the number of children of vertices in the path $[ u(0), u(n)[$  branching
to the left of $[u(0), u(n)[$. \label{P2}
\end{enumerate}

\noindent The discrete genealogical order $\preccurlyeq$  of $\tau$ can be recovered from $W^{\ast}$ (see the proof of \cite[Proposition 1.2]{Leje2005} for details). More precisely,  for $\ast \in \{\uparrow, \downarrow \}$, we consider the sequence $\varnothing = u(0) \prec_{\ast} u(1) \prec_{\ast} \dots \prec_{\ast} u(\zeta(\tau)-1)$ of vertices of $\tau$ in lexicographical or reverse-lexicographical order. For $i, j \in \{0, \dots, \zeta(\tau)-1 \}$, we have that
\begin{eqnarray} \label{eq8}
u(i) \preccurlyeq u(j) \hspace*{4mm} \text{if and only if} \hspace*{4mm} i \leq j \hspace*{4mm} \text{and}  \hspace*{4mm} W^{\ast}(i) = \inf_{i \leq n \leq j} W^{\ast}(n).
\end{eqnarray}

\noindent  In words, $u(i)$ is an ancestor of $u(j)$ that is equivalent to say that $u(i) \in [u(0), u(j)]$. In particular, $u(i) \prec u(j)$ when $u(i) \preccurlyeq u(j)$ and $i \neq j$.

There are two other discrete paths that can be used to code a rooted plane tree $\tau$. Detailed description and properties can be found for example in \cite{Du2003, Du2005}.

\paragraph{The height process.} Let $\varnothing = u(0) \prec_{\uparrow} u(1) \prec_{\uparrow} \dots \prec_{\uparrow} u(\zeta(\tau)-1)$ be the sequence of vertices of $\tau$ in lexicographical order. We define the height process $H = (H(n): 0 \leq n \leq \zeta(\tau) -1)$  by letting $H(n) = |u(n)| = r_{\tau}^{\text{gr}} (\varnothing, u(n))$, for every $n \in \{0, \dots, \zeta(\tau)-1 \}$, where $r_{\tau}^{\text{gr}}$ denotes the graph distance in $\tau$. We also think of $H$ as a continuous function on $[0,  \zeta(\tau)-1]$, obtained by linear interpolation $s \mapsto (1-\{s\})H(\lfloor s \rfloor) + \{s\} H(\lfloor s \rfloor+1)$, where $\{x \} = x - \lfloor x \rfloor$.

\paragraph{The contour process.} Define the contour sequence $\{y(0), y(1), \dots, y(2(\zeta(\tau)-1))\}$ of $\tau$ as follows: $y(0) = \varnothing$ and for each $i \in \{0, \dots, 2(\zeta(\tau)-1)\}$, $y(i+1)$ is either the first child of $y(i)$ which does not appear in the sequence $\{y(0), y(1), \dots, y(i)\}$, or the parent of $y(i)$ if all its children already appear in this sequence. We then define the contour process $C = (C(n), 0 \leq n \leq 2(\zeta(\tau)-1))$ by letting $C(n) = |y(n)| = r_{\tau}^{\text{gr}} (\varnothing , y(n))$, for every $n \in \{0, \dots, 2(\zeta(\tau)-1) \}$. Again, we shall think of $C$ as a continuous function on $[0, 2(\zeta(\tau)-1)]$, obtained by linear interpolation. Note that $C(n) \geq 0$ for every $n \in \{0, \dots, 2(\zeta(\tau)-1) \}$ and $C(0) = C(2(\zeta(\tau)-1))= 0$. \\

Let $\mathbf{t}_{N}$ be an $\alpha$-stable Galton-Watson tree of index $\alpha \in (1,2]$ (i.e.\ a critical GW-trees whose offspring distribution lies in the domain of attraction of an $\alpha$-stable law, conditioned to have total size $N \in \mathbb{N}$). Recall that an $\alpha$-stable Galton-Watson tree is a random element of $\mathbb{T}$; see, for instance, \cite{Du2003, Leje2005}. The asymptotic behaviour of $\mathbf{t}_{N}$ is well understood, in particular through scaling limits of its different coding functions. We write $W_{N}^{\uparrow} = (W_{N}^{\uparrow}( \lfloor N t \rfloor ), t \in [0,1])$, $H_{N} = (H_{N}((N-1) t), t \in [0,1])$ and $C_{N} = (C_{N}(2(N-1) t), t \in [0,1])$ for the (normalized) $\L$ukasiewicz, height process and contour process of $\mathbf{t}_{N}$. Let $X_{\alpha} = (X_{\alpha}(t), t \in [0,1])$ be the so-called normalized excursion of a spectrally positive strictly stable L\'evy process of index $\alpha$ and let $H_{\alpha} = (H_{\alpha}(t), t \in [0,1])$ be its associated continuous excursion height process. We refer to \cite{Du2003} for their precise construction as we will not need them here. The following known results will be useful.

\begin{theorem}[\cite{Al1993}, \cite{Du2003}, \cite{Kortchemski2013}] \label{Theo3}
Let $\mathbf{t}_{N}$ be an $\alpha$-stable Galton-Watson tree. There exists a slowly varying function $\ell: \mathbb{R}_{+} \rightarrow \mathbb{R}_{+}$ such that for $a_{N} = \ell(N) N^{1-\frac{1}{\alpha}}$ and $b_{N} = \ell(N)^{-1}N^{-\frac{1}{\alpha}}$, the triple
\begin{eqnarray}
\left( \left( a_{N} C_{N}( 2(N-1)t ), a_{N} H_{N}( (N-1)t ), b_{N} W_{N}^{\uparrow}(\lfloor Nt \rfloor) \right), t \in [0,1] \right)
\end{eqnarray}

\noindent converges in distribution towards
\begin{eqnarray}
\left( \left( H_{\alpha}(t), H_{\alpha}(t), X_{\alpha}(t) \right), t \in [0,1] \right), \hspace*{4mm} \text{as} \hspace*{2mm} N \rightarrow \infty,
\end{eqnarray}

\noindent in the space $\mathbb{C}([0,1], \mathbb{R}) \times \mathbb{C}([0,1], \mathbb{R}) \times \mathbb{D}([0,1], \mathbb{R})$, where $\mathbb{C}([0,1], \mathbb{R})$ is the space of real-valued continuous functions on $[0,1]$ equipped with uniform topology and $\mathbb{D}([0,1], \mathbb{R})$ is the space of real-valued c\`adl\`ag functions on $[0,1]$ equipped with the Skorohod ${\rm J}_{1}$ topology.
\end{theorem}

In the particular case $\alpha =2$, we have that $(H_{2}, X_{2}) = (2 X_{2}, X_{2})$, and $X_{2}$ is the standard normalized Brownian excursion. On the other hand, the function $\ell$ is exactly the same that appears in (\ref{eq1}).

Let $W_{N}^{\downarrow} = (W_{N}^{\downarrow}( \lfloor N t \rfloor ), t \in [0,1])$ denote the (normalized) reverse-$\L$ukasiewicz of $\mathbf{t}_{N}$. The following result is a consequence of Theorem \ref{Theo3}.

\begin{lemma}[{\cite[Lemma 2.4 and Lemma 4.2]{Dieu2015}}] \label{Cor3}
Let $\mathbf{t}_{N}$ be an $\alpha$-stable GW-tree and let $(b_{N})_{N \in \mathbb{N}}$ be as in Theorem \ref{Theo3}. There exists a process $\tilde{X}_{\alpha} = (\tilde{X}_{\alpha}(t),  t \in[0, 1])$ with the same distribution as $X_{\alpha}$ such that, jointly with the convergence of Theorem \ref{Theo3}, we have that, in distribution,
\begin{eqnarray}
\left( \left( b_{N} W_{N}^{\uparrow}(\lfloor Nt \rfloor), b_{N} W_{N}^{\downarrow}(\lfloor Nt \rfloor) \right), t \in [0,1]  \right) \rightarrow \left( \left(  X_{\alpha}(t), \tilde{X}_{\alpha}(t)  \right), t \in[0,1]  \right), \hspace*{4mm} \text{as} \hspace*{2mm} N \rightarrow \infty,
\end{eqnarray}
\noindent in the space $\mathbb{D}([0,1], \mathbb{R}) \times \mathbb{D}([0,1], \mathbb{R})$. Moreover,
\begin{enumerate}[label=(\textbf{\roman*})]
\item For $\alpha = 2$, $\tilde{X}_{2} = (X_{2}(1-t), t \in [0,1])$.

\item For $\alpha \in (1,2)$ and every jump-time $t$ of $X_{\alpha}$, 
\begin{align*}
\Delta_{t} \coloneqq X_{\alpha}(t) - X_{\alpha}(t-) = \tilde{X}_{\alpha}(1- t-\gamma(t)) - \tilde{X}_{\alpha}((1- t-\gamma(t))-)
\end{align*}
\noindent a.s., where $\gamma(t) \coloneqq \inf\{ s \in (t, 1]: X_{\alpha}(s) = X_{\alpha}(t-) \} - t$.
\end{enumerate}
\end{lemma}

\subsection{The $\alpha$-stable L\'evy tree}  \label{Sec13}

This section recalls the excursion representation of the $\alpha$-stable L\'evy tree of index $\alpha \in (1,2]$ and some of its well-established properties. Let $H_{\alpha} =(H_{\alpha}(t), t \in [0,1])$ be the excursion height process associated to $X_{\alpha}$. Recall that $H_{\alpha}$ is continuous, satisfies $H_{\alpha}(0) = H_{\alpha}(1)=0$ and $H_{\alpha}(t) > 0$ for all $t \in (0,1)$; see \cite{Du2003}. The $\alpha$-stable L\'evy tree is the random compact rooted measure $\mathbb{R}$-tree coded by the excursion height process $H_{\alpha}$ as follows. Consider the pseudo-distance on $[0,1]$,
\begin{eqnarray}
r_{\alpha}(s,t) \coloneqq H_{\alpha}(s) + H_{\alpha}(t) - 2 \inf_{u \in [s \wedge t, s \vee t]} H_{\alpha}(u), \hspace*{4mm} \text{for} \hspace*{2mm} s,t \in [0,1],
\end{eqnarray}

\noindent and define an equivalence relation on $[0,1]$ by setting $s \sim_{\alpha} t$ if and only if $r_{\alpha}(s,t) = 0$. The image of the projection $p_{\alpha}:[0,1] \rightarrow [0,1] \setminus \sim_{\alpha}$ endowed with the push forward of $r_{\alpha}$ (again denoted $r_{\alpha}$), i.e.\ $(\mathcal{T}_{\alpha}, r_{\alpha}, \rho_{\alpha}) := (p_{\alpha}([0,1]), r_{\alpha}, p_{\alpha}(0))$, is the compact rooted $\mathbb{R}$-tree; see Evans and Winter \cite[Lemma 3.1]{EW2006}. Furthermore, $\mathcal{T}_{\alpha}$ is endowed with the probability measure $\mu_{\alpha} := (p_{\alpha})_{\ast} {\rm Leb}_{[0,1]}$ (the push forward of the Lebesgue measure ${\rm Leb}$ on $[0,1]$). Finally, $\mathcal{T}_{\alpha} = (\mathcal{T}_{\alpha}, r_{\alpha}, \rho_{\alpha}, \mu_{\alpha})$ is the rooted measure $\mathbb{R}$-tree known as the $\alpha$-stable L\'evy tree. Recall that the support of $\mu_{\alpha}$ is the set of leaves of $\mathcal{T}_{\alpha}$; see \cite[Theorem 4.6]{Du2005}. 

The genealogical order $\preccurlyeq$ in $\mathcal{T}_{\alpha}$ is defined for $x, y \in \mathcal{T}_{\alpha}$ by letting $x \preccurlyeq y$ if $x \in [\rho_{\alpha}, y]$, where $\rho_{\alpha}$ is the root of $\mathcal{T}_{\alpha}$. In particular, if for $x \neq y$, we write $x \prec y$. If $x,y \in \mathcal{T}_{\alpha}$  there exists a unique $z \in \mathcal{T}_{\alpha}$ such that $[\rho_{\alpha}, x] \cap [\rho_{\alpha}, y] = [\rho_{\alpha}, z]$, then $z$ is called the most recent common ancestor of $x$ and $y$, and it is denoted by $z = x \wedge y$.

As in the discrete case, the genealogical order $\preccurlyeq$ of the $\alpha$-stable L\'evy tree can also be recovered from the path $X_{\alpha} = (X_{\alpha}(t), t \in [0,1])$. We define a partial order on $[0,1]$, denoted by $\preccurlyeq$, which is compatible with the projection $p_{\alpha} :[0,1] \rightarrow \mathcal{T}_{\alpha}$ by setting, for every $s,t \in [0,1]$,
\begin{eqnarray}
s \preccurlyeq t \hspace*{5mm} \text{if} \hspace*{5mm} s \leq t \hspace*{4mm} \text{and} \hspace*{4mm} X_{\alpha}(s-) \leq I_{\alpha}(s,t),
\end{eqnarray}

\noindent where $I_{\alpha}(s,t) = \inf_{u \in [s, t]} X_{\alpha}(u)$ for $0 \leq s \leq t \leq 1$. By convention $X_{\alpha}(0-) = 0$. In particular, $s \prec t$ when $s \preccurlyeq t$ and $s \neq t$. It is a simple matter to check that $\preccurlyeq$ is indeed a partial order which is compatible with the genealogical order on $\mathcal{T}_{\alpha}$, meaning that a point $x \in \mathcal{T}_{\alpha}$ is an ancestor of $y \in \mathcal{T}_{\alpha}$ if and only if there exist $s \preccurlyeq t \in [0,1]$ with $x = p_{\alpha}(s)$ and $y=p_{\alpha}(t)$ (equivalently, $x \in [\rho_{\alpha}, y]$ with $\rho_{\alpha} = p_{\alpha}(0)$). For every $s,t \in [0,1]$, let $s \wedge t$ be the most recent common ancestor (for the relationship $\preccurlyeq$ on $[0,1]$) of $s$ and $t$. Then $p_{\alpha}(s \wedge t)$ also is the most recent common ancestor of $p_{\alpha}(s)$ and $p_{\alpha}(t)$ in $\mathcal{T}_{\alpha}$.

For $\alpha \in (1,2)$, the branching-points of $\mathcal{T}_{\alpha}$ are in one-to-one correspondence with the jumps of $X_{\alpha}$; see \cite[Proposition 2]{Miermont2005}. More precisely, write $\Delta_{t} \coloneqq X_{\alpha}(t) -  X_{\alpha}(t-)$, for $t \in (0,1]$, and $\Delta_{0} \coloneqq 0$. Then, $x \in \mathcal{T}_{\alpha}$ is a branching-point if and only if there exists a unique $t \in [0, 1]$ such that $x= p_{\alpha}(t)$ and $\Delta_{t}  > 0$. In particular, for any $x=p_{\alpha}(t)$ a branching-point, one can recover the jump $\Delta_{t}$ as the local time (or width) of $x$ defined in (\ref{eq18}).

\subsection{Further properties}  \label{Sec14}

In this section, we derive additional asymptotic properties of the reverse-$\L$ukasiewicz paths of the $\alpha$-stable GW-tree $\mathbf{t}_{N}$.

Let $u(0) \prec_{\uparrow} u(1) \prec_{\uparrow} \dots \prec_{\uparrow} u(N-1)$ (resp.\ $u(0) \prec_{\downarrow} u(1) \prec_{\downarrow} \dots \prec_{\downarrow} u(N-1)$) be the vertices of $\mathbf{t}_{N}$, listed in lexicographical order (resp.\ in reverse-lexicographical order). Recall the definition of the discrete genealogical order $\preccurlyeq$ in \eqref{eq8}. In this section, we will, with a slight abuse of notation, write $n  \preccurlyeq m$ (resp.\ $n \prec m$) to denote $u(n)  \preccurlyeq u(m)$ (resp.\ $u(n) \prec u(m)$), for $n,m = 0, \dots, N-1$. Define $\sigma_{N}^{\uparrow}(1) \coloneqq 0$ and for $t \in [0,1)$, 
\begin{eqnarray} \label{eq24}
\sigma_{N}^{\uparrow}(t) \coloneqq \sideset{}{^{\uparrow}}\sum_{0 \preccurlyeq k \preccurlyeq \lfloor Nt \rfloor} \Delta W^{\uparrow}_{N}(k+1) - \Delta W^{\uparrow}_{N}(\lfloor Nt \rfloor+1) \wedge 0,
\end{eqnarray}
\noindent where the sum $\sum_{0 \preccurlyeq k \preccurlyeq \lfloor Nt \rfloor}^{\uparrow}$ is over all ancestors of the vertex $u(\lfloor Nt \rfloor)$, with respect to the lexicographical order.
Note that, for $t \in [0,1)$, $\sigma_{N}^{\uparrow}(t)$ is the sum of the quantities $c(v)-1$ over all ancestors $v \in \mathbf{t}_{N}$ of $u(\lfloor Nt \rfloor)$ that are not a leaf, with respect to the lexicographical order. Recall that $c(v)$ denotes the number of children of the vertex $v$. Note that $u(\lfloor Nt \rfloor)$ is the only possible leaf among all the ancestors of $u(\lfloor Nt \rfloor)$.

By the Skorokhod's representation theorem (see, e.g., \cite[Theorem 6.7 in Chapter 1]{Bi1999}), we can and will assume for the remainder of this section that we are working in a probability space where the convergences of Theorem \ref{Theo3} and Lemma \ref{Cor3} hold jointly almost surely.

The first result of this section establishes that, provided $\mathbf{t}_{N}$ is a $2$-stable GW-tree, $b_{N}\sigma_{N}^{\uparrow}$ converges uniformly to $2X_{2}$. 
\begin{theorem} \label{Theo4}
Let $\mathbf{t}_{N}$ be a $2$-stable GW-tree and let $(b_{N})_{N \in \mathbb{N}}$ be as in Theorem \ref{Theo3}. Then, 
\begin{align}
\sup_{t \in [0,1]} \left|b_{N}\sigma_{N}^{\uparrow}(t) - 2 X_{2}(t) \right| \rightarrow 0, \hspace*{4mm} \text{as} \hspace*{2mm}  N \rightarrow \infty, \hspace*{2mm} \text{a.s.} 
\end{align}
\end{theorem}

In preparation for the proof of Theorem \ref{Theo4}, we introduce some useful notation and make some important remarks. Observe that the $(k+1)$-th vertex in lexicographical order $u(k) \in \mathbf{t}_{N}$, for $k=0, \dots, N-1$, corresponds to the $\tilde{k}$-th vertex in reverse-lexicographical order, where
\begin{eqnarray}
\tilde{k} = N-1 - k + |u(k)| - D(k),
\end{eqnarray}
\noindent and 
\begin{align}
D(k) \coloneqq \inf \{ m \in \{k+1, \dots, N\}: W^{\uparrow}_{N}(m) - W^{\uparrow}_{N}(k) = -1 \} - (k+1)
\end{align}
\noindent is the number of strict descendants of the $k$-th vertex in lexicographical order. We adopt the convention that $D(N) \coloneqq 0$. Moreover, we have that
\begin{eqnarray}
\Delta W^{\downarrow}_{N}(\tilde{k} +1 ) = \Delta W^{\uparrow}_{N}(k+1), \hspace*{4mm} \text{for} \hspace*{2mm}  k \in \{0,\dots, N-1\}.
\end{eqnarray}

\noindent  For $0 \leq i < j \leq 	N-1$ and $\ast \in \{\uparrow, \downarrow\}$, set
\begin{eqnarray}
\bar{x}_{i,j}^{\ast} \coloneqq  W^{\ast}_{N}(i+1) - I^{\ast}_{N}(i+1,j) \hspace*{4mm} \text{and} \hspace*{4mm} \underline{x}_{i,j}^{\ast} \coloneqq  I^{\ast}_{N}(i+1,j) -W^{\ast}_{N}(i),
\end{eqnarray}

\noindent where $I^{\ast}_{N}(i+1,j) \coloneqq \inf_{i+1 \leq m \leq j}W^{\ast}_{N}(m)$. Note that for any $0 \leq i < j \leq N-1$ such that $u(0) \preceq u(i) \prec u(j)$, the term $\bar{x}_{i,j}^{\uparrow}$ (resp.\ $\bar{x}_{i,j}^{\downarrow}$) corresponds to the number of children of $u(i)$  branching to the left of $[u(0), u(j)[$ (resp.\ branching to the left of $[u(0), u(j)[$), and $\underline{x}_{i,j}^{\uparrow}$ (resp.\ $\underline{x}_{i,j}^{\downarrow}$) is the number of children of $u(i)$ branching to the right of $[u(0), u(j)[$ (resp.\  branching to the right of $[u(0), u(j)[$).

Then, from properties \ref{P1} and \ref{P2}, we deduce the following useful identities: for $t \in [0,1)$,
\begin{eqnarray} \label{eq10}
\sideset{}{^{\uparrow}}\sum_{0 \preccurlyeq k \prec \lfloor Nt \rfloor} \underline{x}_{k,\lfloor N t \rfloor}^{\uparrow} = W_{N}^{\uparrow}(\lfloor N t \rfloor),
\end{eqnarray}
\noindent and
\begin{eqnarray} \label{eq11}
\sideset{}{^{\uparrow}}\sum_{0 \preccurlyeq k \prec \lfloor Nt \rfloor} \bar{x}_{k,\lfloor Nt \rfloor}^{\uparrow} = \sideset{}{^{\downarrow}}\sum_{0 \preccurlyeq k \prec \lfloor Nt \rfloor}  \underline{x}_{k,\tilde{m}(N,t)}^{\downarrow} = W_{N}^{\downarrow}(\tilde{m}(N,t)),
\end{eqnarray}
\noindent where 
\begin{align} \label{eq31}
\tilde{m}(N,t) \coloneqq N-1-\lfloor Nt \rfloor + |u(\lfloor Nt \rfloor)| - D(\lfloor Nt \rfloor).
\end{align}
\noindent Here the sum $\sum_{0 \preccurlyeq k \prec \lfloor Nt \rfloor}^{\uparrow}$ (resp.\ $\sum_{0 \preccurlyeq k \prec \lfloor Nt \rfloor}^{\downarrow}$) is over all strict ancestors of the vertex $u(\lfloor Nt \rfloor)$, with respect to the lexicographical order (resp.\ with respect to the reverse-lexicographical order). Then, for $t \in [0,1)$,
\begin{eqnarray} \label{eq12}
\sigma_{N}^{\uparrow}(t) =  W_{N}^{\uparrow}(\lfloor N t \rfloor) + W_{N}^{\downarrow}(\tilde{m}(N,t)) + R_{N}(t),
\end{eqnarray}
\noindent where 
\begin{align} \label{eq41}
R_{N}(t) \coloneqq \begin{cases}
\Delta W^{\uparrow}_{N}(\lfloor Nt \rfloor+1) - \Delta W^{\uparrow}_{N}(\lfloor Nt \rfloor+1) \wedge 0, & \quad \text{if} \, \, t \in [0,1), \\
0. & \quad \text{if} \, \,  t =0.
\end{cases}
\end{align}

Before proving Theorem \ref{Theo4}, we will need a technical lemma. The proof relies on properties of the Skorohod ${\rm J}_{1}$  topology for which we could not find a reference, so we have stated and proven them in Appendix \ref{Append1}. Recall the definition of $\gamma(t)$ in Lemma \ref{Cor3}.
\begin{lemma} \label{lemma4}
Let $\mathbf{t}_{N}$ be an $\alpha$-stable GW-tree and let $(t_{N})_{N \in \mathbb{N}}$ be any sequence of real numbers in $[0,1]$ such that $t_{N} \rightarrow t \in [0,1]$, as $N \rightarrow \infty$. Then,
\begin{enumerate}[label=(\textbf{\roman*})]
\item \label{lemma4C1} If $W^{\uparrow}_{N}(\lfloor N t_{N} \rfloor) \rightarrow X_{\alpha}(t-)$ as $N\rightarrow \infty$, then 
$N^{-1}D(\lfloor N t_{N} \rfloor)  \rightarrow \tilde{\gamma}(t-)$ a.s., where
\begin{align}
\tilde{\gamma}(t-) \coloneqq \inf \{ s \in (0, 1]: X_{\alpha}(t+s)-X_{\alpha}(t-) <0 \},
\end{align}
\noindent with the convention $\tilde{\gamma}(1-) \coloneqq 0$. Moreover, $\tilde{\gamma}(0-)=1$ and for $t \in (0,1)$, $\tilde{\gamma}(t-)=\gamma(t)$ if $\Delta_{t} > 0$, and $\tilde{\gamma}(t-)=0$, otherwise.  
\item \label{lemma4C1} If $W^{\uparrow}_{N}(\lfloor Nt_{N} \rfloor) \rightarrow X_{\alpha}(t)$ as  $N\rightarrow \infty$, then $N^{-1}D(\lfloor N t_{N} \rfloor)  \rightarrow \tilde{\gamma}(t)$ a.s., where
\begin{align}
\tilde{\gamma}(t) \coloneqq \inf \{ s \in (0, 1]: X_{\alpha}(t+s)-X_{\alpha}(t) <0 \},
\end{align}
\noindent with the convention $\tilde{\gamma}(1) \coloneqq 0$. Moreover, $\tilde{\gamma}(0)=1$ and for $t \in (0,1)$, $\tilde{\gamma}(t)=0$.  
\end{enumerate}
\end{lemma}

\begin{proof}
Note that $D(k) = \tilde{D}(k)-1$, where 
\begin{align}
\tilde{D}(k) \coloneqq \inf \{ m \in \{1, \dots, N\}: W^{\uparrow}_{N}((k+m)\wedge N) - W^{\uparrow}_{N}(k) <0 \},
\end{align}
\noindent for $k=0, \dots, N$, with the convention $\tilde{D}(N) \coloneqq 1$. Then, 
\begin{align}
N^{-1}\tilde{D}(\lfloor N t_{N} \rfloor) = \inf \{ s \in [N^{-1},1]: W^{\uparrow}_{N}(\lfloor N \theta_{N}(s) \rfloor) - W^{\uparrow}_{N}(\lfloor Nt_{N} \rfloor) <0 \},
\end{align}
\noindent where $\theta_{N}(s)\coloneqq (\frac{\lfloor Ns \rfloor+ \lfloor N t_{N} \rfloor}{N})\wedge 1$, for $s \in [0,1]$.  

Observe that $X_{\alpha}$ satisfies \ref{H1}-\ref{H4}. For $\alpha \in (1,2)$, this is proven in \cite[Proposition 2.10]{Kortchemski20142}, and for $\alpha =2$, these properties are clearly satisfied. Therefore, our claims follow from Lemmas  \ref{lemma1}, \ref{lemma2} and \ref{lemma3} (note that condition \eqref{eq28} is clearly satisfied).
\end{proof}

\begin{proof}[Proof of Theorem \ref{Theo4}]
The proof proceeds by contradiction. Suppose that there exists $\varepsilon >0$ and a sequence of real numbers $(t_{N})_{N \in \mathbb{N}}$ in $[0,1]$ such that for every $N$ sufficiently large
\begin{eqnarray} \label{eq13}
\left | b_{N} \sigma_{N}^{\uparrow}(t_{N}) - 2 X_{2}(t_{N}) \right| \geq \varepsilon.
\end{eqnarray}
\noindent By compactness, we may assume without loss of generality that $t_{N} \rightarrow t \in [0,1]$, as $N \rightarrow \infty$. 

Clearly, $b_{N}R_{N}(t) \rightarrow 0$, as $N \rightarrow \infty$, almost surely. Next, let $h_{N} \coloneqq \max \{ |v|: v \in \mathbf{t}_{N} \}$ be the height of $ \mathbf{t}_{N}$. By Theorem \ref{Theo3}, we have that, as $N \rightarrow \infty$,
\begin{align} \label{eq9}
N^{-1} h_{N} \rightarrow 0,  \hspace*{2mm} \text{a.s.}
\end{align}
\noindent  Then, by Lemma \ref{lemma4} and \eqref{eq9}, for $t \in [0,1]$, $N^{-1}\tilde{m}(N,t) \rightarrow 1-t-\mathbf{1}_{\{0\}}(t)$, as $N \rightarrow \infty$, almost surely. Now, for $t \in (0,1]$, it follows from (\ref{eq12}), Theorem \ref{Theo4} and Lemma \ref{Cor3}, that $b_{N} \sigma_{N}^{\uparrow}(t_{N}) \rightarrow X_{2}(t) + \tilde{X}_{2}(1-t)=2X_{2}(t)$, as $N \rightarrow \infty$, almost surely. Similarly, for $t =0$, $b_{N} \sigma_{N}^{\uparrow}(t_{N}) \rightarrow X_{2}(0) + \tilde{X}_{2}(0)=2X_{2}(0)$, as $N \rightarrow \infty$, almost surely. This leads to the desired contradiction, which concludes our proof.
\end{proof}

To conclude this section, we present a technical estimate that will be useful in the next section for the proof of our main result, Theorem \ref{NewTheo}.

Let $(t_{i}, i \in \mathbb{N})$ be a sequence of independent random uniform variables on $[0,1]$. Assume that they are also independent of $W_{N}^{\uparrow}$, $W_{N}^{\downarrow}$, $H_{N}$, $C_{N}$, $X_{\alpha}$ and $H_{\alpha}$. For $i \geq 1$, set $U_{i,N} \coloneqq u(\lfloor t_{i} N \rfloor)$ such that $U_{i,N}$ is the $\lfloor t_{i} N \rfloor$-th vertex of $\mathbf{t}_{N}$ in lexicographical order (and in particular, $U_{i,N}$ is a uniform random vertex of $\mathbf{t}_{N}$). Recall the definition of the pruning measure $\nu_{N}^{\rm bra}$ from \eqref{eq7}. Recall also that, for $k \in \mathbb{N}$, $\mathbf{U}_{k,N} \rrbracket$ denotes the subtree of $\mathbf{t}_{N}$ spanned by its root and the vertices $\mathbf{U}_{k,N} = (U_{1,N}, \dots, U_{k,N})$. 

\begin{lemma} \label{lemma7} 
Let $\mathbf{t}_{N}$ be an $\alpha$-stable GW-tree, for $\alpha \in (1,2)$, and fix $k \in \mathbb{N}$. Then, for all $\varepsilon>0$, there exists $\beta >0$ such that 
\begin{eqnarray} \label{eq6}
\limsup_{N \rightarrow \infty} \nu_{N}^{\rm bra} ( \{ v \in \llbracket \mathbf{U}_{k,N} \rrbracket: b_{N}(c(v)-1) \leq \beta\})  \leq \varepsilon, \hspace*{2mm} \text{a.s.}
\end{eqnarray}
\end{lemma}

\begin{proof}
Observe that 
\begin{align}
\nu_{N}^{\rm bra} ( \{ v \in \llbracket \mathbf{U}_{k,N} \rrbracket: b_{N}(c(v)-1) \leq \beta\}) \leq \sum_{i=1}^{k} \nu_{N}^{\rm bra} ( \{ v \in \llbracket \mathbf{U}_{i,N} \rrbracket: b_{N}(c(v)-1) \leq \beta\}) 
\end{align}
\noindent Fix $\varepsilon >0$. Then, to establish our claim, it is sufficient to prove that for each $i \in \{1, \dots, k\}$,
\begin{eqnarray}  \label{eq37}
\limsup_{N \rightarrow \infty} \nu_{N}^{\rm bra} ( \{ v \in \llbracket \mathbf{U}_{i,N} \rrbracket: b_{N}(c(v)-1) \leq \beta\})  \leq \varepsilon /k, \hspace*{2mm} \text{a.s.}
\end{eqnarray}

The proof of \eqref{eq37} draws inspiration from ideas presented in the proof of \cite[Theorem 4.1]{Igor2014}. Recall the definition of $R_{N}(t_{i})$ in \eqref{eq41}. Then, note that
\begin{align} \label{eq38}
& \nu_{N}^{\rm bra} ( \{ v \in \llbracket \mathbf{U}_{i,N} \rrbracket: b_{N}(c(v)-1) \leq \beta\}) \nonumber \\
& \quad \quad = b_{N} \sideset{}{^{\uparrow}}\sum_{0 \preccurlyeq k \preccurlyeq \lfloor N t_{i} \rfloor} \Delta W^{\uparrow}_{N}(k+1) \mathbf{1}_{ \left\{ \Delta W^{\uparrow}_{N}(k+1) \leq \beta b_{N}^{-1} \right\}}  \nonumber \\
& \quad \quad  \quad \quad  - b_{N}\left( \Delta W^{\uparrow}_{N}(\lfloor Nt_{i} \rfloor+1) \wedge 0 \right) \mathbf{1}_{ \left\{ \Delta W^{\uparrow}_{N}(\lfloor Nt_{i} \rfloor+1) \leq \beta b_{N}^{-1} \right\}}  \nonumber \\
& \quad \quad = b_{N}\sideset{}{^{\uparrow}}\sum_{0 \preccurlyeq k \prec \lfloor Nt_{i} \rfloor} \left( \bar{x}_{k,\lfloor N t_{i} \rfloor}^{\uparrow}  + \underline{x}_{k,\lfloor N t_{i} \rfloor}^{\uparrow}
 \right) \mathbf{1}_{ \left\{ \Delta W_{N}^{\uparrow}(k+1) \leq \beta b_{N}^{-1} \right\}} +  b_{N} \tilde{R}_{N}(t_{i}),
\end{align}
\noindent where 
\begin{align}
\tilde{R}_{N}(t_{i}) \coloneqq R_{N}(t_{i}) \mathbf{1}_{ \left\{ \Delta W^{\uparrow}_{N}(\lfloor N t_{i} \rfloor+1) \leq \beta b_{N}^{-1} \right\}}.
\end{align}
\noindent From \eqref{eq10} and \eqref{eq11}, we get that
\begin{align} \label{eq32} 
& b_{N}  \sideset{}{^{\uparrow}}\sum_{0 \preccurlyeq k \prec \lfloor Nt_{i} \rfloor} \underline{x}_{k,\lfloor N t_{i} \rfloor}^{\uparrow} \mathbf{1}_{ \left\{ \Delta W_{N}^{\uparrow}(k+1) \leq \beta b_{N}^{-1} \right\}} \nonumber \\
& \quad \quad \quad \leq b_{N}  W_{N}^{\uparrow}(\lfloor N t_{i} \rfloor) - b_{N}  \sideset{}{^{\uparrow}}\sum_{0 \preccurlyeq k \prec \lfloor Nt_{i} \rfloor}  \underline{x}_{k,\lfloor N t_{i} \rfloor}^{\uparrow} \mathbf{1}_{ \left\{ \Delta W_{N}^{\uparrow}(k+1) > \beta b_{N}^{-1} \right\}}
\end{align}
\noindent and
\begin{align} \label{eq33}
& b_{N}  \sideset{}{^{\uparrow}}\sum_{0 \preccurlyeq k \prec \lfloor Nt_{i} \rfloor} \bar{x}_{k,\lfloor N t_{i} \rfloor}^{\uparrow} \mathbf{1}_{ \left\{ \Delta W_{N}^{\uparrow}(k+1) \leq \beta b_{N}^{-1} \right\}} \nonumber \\
& \quad \quad \quad \leq b_{N}  W_{N}^{\downarrow}(\tilde{m}(N,t_{i})) - b_{N} \sideset{}{^{\downarrow}}\sum_{0 \preccurlyeq k \prec \lfloor Nt_{N} \rfloor} \underline{x}_{k,\tilde{m}(N,t_{i})}^{\downarrow} \mathbf{1}_{ \left\{ \Delta W_{N}^{\downarrow}(k+1) > \beta b_{N}^{-1} \right\}},
\end{align}
\noindent where $\tilde{m}(N,t_{i})$ is defined in \eqref{eq31}. 

Note that we can always choose $\beta>0$ small enough such that
\begin{eqnarray}
\limsup_{N \rightarrow \infty} b_{N} \tilde{R}_{N}(t_{i}) < \varepsilon /3k.
\end{eqnarray}
\noindent Then, by \eqref{eq38}, \eqref{eq32} and \eqref{eq33}, to prove (\ref{eq37}), we need to show that the right-hand sides of (\ref{eq32}) and (\ref{eq33}) are both less than $\varepsilon/3k$ for $\beta>0$ small enough. 

First, we prove that the right-hand side of (\ref{eq32}) is less that $\varepsilon/3k$. It is known that $\sum_{0 \preccurlyeq s \preccurlyeq t_{i}} ( I_{\alpha}(s,t_{i}) - X_{\alpha}(s-)) = X_{\alpha}(t_{i})$; see for e.g., \cite[Corollary 3.4]{Igor2014}. Thus, we can assume that $\beta > 0$ has been chosen sufficiently small so that
\begin{align} \label{eq34}
\sum_{0 \preccurlyeq s \preccurlyeq t_{i}} ( I_{\alpha}(s,t_{i}) - X_{\alpha}(s-)) \mathbf{1}_{\{ \Delta_{s} > \beta \}} \geq X_{\alpha}(t_{i}) - \varepsilon /3k.
\end{align}
\noindent Moreover, since $t_{i}$ is a.s.\ a continuity point of $X_{\alpha}$, Theorem \ref{Theo3} implies that, as $N \rightarrow \infty$,
\begin{align} \label{eq35}
b_{N}W_{N}^{\uparrow}(\lfloor N t_{i} \rfloor)  \rightarrow  X_{\alpha}(t_{i}), \quad  \text{a.s.}
\end{align}
\noindent On the other hand, recall that $\{ s \in [0,1]: 0 \preccurlyeq s \preccurlyeq t_{i} \hspace*{2mm} \text{and} \hspace*{2mm} \Delta_{s} > 0  \}$ is a countable set; see \cite[Theorem 4.6]{Du2005}.  In particular, for each fixed $\beta >0$, the set $\{ s \in [0,1]: 0 \preccurlyeq s \preccurlyeq t_{i} \hspace*{2mm} \text{and} \hspace*{2mm} \Delta_{s} > \beta  \}$ is finite. Recall that properties of the Skorohod topology entail that the jumps of $(b_{N}W^{\uparrow}_{N}( \lfloor N t \rfloor), t \in [0,1])$ converge to the jumps of $X_{\alpha}$, together with their locations; see for e.g., \cite[Proposition 2.7 in Chapter VI]{Jacod2003}. Then, as $N \rightarrow \infty$,
\begin{align} \label{eq36}
b_{N} \sideset{}{^{\uparrow}}\sum_{0 \preccurlyeq k \prec \lfloor Nt_{i} \rfloor}  \underline{x}_{k,\lfloor N t_{i} \rfloor}^{\uparrow} \mathbf{1}_{ \left\{ \Delta W_{N}^{\uparrow}(k+1) > \beta b_{N}^{-1} \right\}} \rightarrow \sum_{0 \preccurlyeq s \preccurlyeq t_{i}} ( I_{\alpha}(s,t_{i}) - X_{\alpha}(s-)) \mathbf{1}_{\{ \Delta_{s} > \beta \}};
\end{align}
\noindent (recall that $t_{i}$ is a.s.\ a continuity point of $X_{\alpha}$ and thus $\Delta_{t_{i}}=0$). Therefore, combining \eqref{eq34}, \eqref{eq35} and \eqref{eq36}, we deduce that for sufficiently large $N$, the right-hand side of (\ref{eq32}) is less than $\varepsilon/3k$.  

The proof that the right-hand side of (\ref{eq33}) is less than $\varepsilon/3k$  for sufficiently small $\beta>0$ is similar. It follows from \eqref{eq9}, \eqref{eq35} and Lemma \ref{lemma4} that, as $N \rightarrow \infty$,
\begin{align} \label{eq39}
\tilde{m}(N,t_{i}) \rightarrow 1-t_{i}, \quad  \text{a.s.}
\end{align}
\noindent Note that $1-t_{i}$ is a.s.\ a continuity point of $\tilde{X}_{\alpha}$. Thus, Lemma \ref{Cor3} implies that, as $N \rightarrow \infty$,
\begin{align} \label{eq40}
b_{N}W_{N}^{\uparrow}(\tilde{m}(N,t_{i}))  \rightarrow  \tilde{X}_{\alpha}(1-t_{i}), \quad  \text{a.s.}
\end{align}
\noindent Therefore, using \eqref{eq39} and \eqref{eq40}, one can readily deduce that for sufficiently large $N$, the right-hand side of (\ref{eq33}) is less than $\varepsilon/3k$, in the same way we proved that the right-hand side of (\ref{eq32}) is less than $\varepsilon/3k$. 

This concludes the proof of Lemma \ref{lemma7}.
\end{proof}

\section{Proof of Theorem \ref{NewTheo}}  \label{Sec15}

In this section, we prove Theorem \ref{NewTheo}. Throughout the section, we let $(a_{N})_{N \in \mathbb{N}}$ and $(b_{N})_{N \in \mathbb{N}}$ be the sequences of positive real numbers as in  Theorem \ref{Theo3} or equivalently, defined in (\ref{eq23}). The proof of Theorem~\ref{NewTheo} relies on the following result: the LWV-convergence of a properly rescaled $\alpha$-stable GW-tree equipped with its pruning measure to the $\alpha$-stable L\'evy tree equipped with its corresponding pruning measure.

\begin{proposition} \label{Prop1}
For $N \in \mathbb{N}$ and $\ast \in \{ {\rm ske}, {\rm bra}, {\rm mix}\}$, let $\mathbf{t}_{N}^{\ast} = (\mathbf{t}_{N}, a_{N} \cdot r_{N}^{{\rm gr}}, \rho_{N}, \mu_{N}^{\rm nod}, \nu_{N}^{\ast})$ be the bi-measure $\mathbb{R}$-tree associated with an $\alpha$-stable ${\rm GW}$-tree. Similarly, let $\mathcal{T}_{\alpha}^{\ast} = (\mathcal{T}_{\alpha}, r_{\alpha}, \rho_{\alpha}, \mu_{\alpha}, \nu^{\ast}_{\alpha})$ be the bi-measure $\mathbb{R}$-tree associated with an $\alpha$-stable L\'evy tree. Then, $\mathbf{t}_{N}^{\ast} \rightarrow \mathcal{T}^{\ast}_{\alpha}$, in distribution, as $N \rightarrow \infty$, with respect to the LWV-topology.
\end{proposition}

We have now all the ingredients to prove Theorem \ref{NewTheo}.
\begin{proof}[Proof of Theorem \ref{NewTheo}]
The result follows from Proposition \ref{Prop1} and \cite[Lemma 3.1, Lemma 3.3 and Theorem 3.6]{LohrVoisinWinter2015} (i.e., the fact that the pruning process is a strong Markov $\mathbb{H}_{f,  \sigma}$-valued process, with c\`adl\`ag paths, and whose distribution is weakly continuous in the initial value).
\end{proof}

The rest of this section is dedicated to the proof of Proposition \ref{Prop1}. 

\begin{proof}[Proof of Proposition \ref{Prop1}]
By the Skorokhod representation theorem, we work in a probability space where the convergences in Theorem \ref{Theo3}, Lemma \ref{Cor3} and Theorem \ref{Theo4} hold almost surely. Theorem \ref{Theo3} implies the convergence of the rescaled $\alpha$-stable GW-tree to the $\alpha$-stable L\'evy tree $\mathcal{T}_{\alpha}$ stated in (\ref{eq1}) holds a.s.\ with respect to the Gromov-Hausdorff-weak topology; see e.g.\ \cite[Proposition 2.9]{ADH2014}. In particular, the same convergence holds a.s.\ with respect to the Gromov-weak topology.

Recall that the pruning measures $\nu^{\rm ske}_{N}$ and $\nu_{\alpha}^{\rm ske}$ are the length measures on the rescaled  $\alpha$-stable GW-tree $a_{N} \cdot \mathbf{t}_{N} = (\mathbf{t}_{N}, a_{N} \cdot r_{N}^{\text{gr}}, \rho_{N}, \mu_{N}^{\rm nod})$ and $\alpha$-stable L\'evy tree $\mathcal{T}_{\alpha} = (\mathcal{T}_{\alpha}, r_{\alpha}, \rho_{\alpha}, \mu_{\alpha})$, respectively. According to the definition before \cite[Example 2.24]{LohrVoisinWinter2015}, the length measures  $\nu^{\rm ske}_{N}$ and $\nu_{\alpha}^{\rm ske}$ depend continuously on the distances of the corresponding  underlying trees. Then \cite[Proposition 2.25]{LohrVoisinWinter2015} and the Gromov-weak convergence of $a_{N} \cdot \mathbf{t}_{N}$ to $\mathcal{T}_{\alpha}$ imply that, as $N \rightarrow \infty$,
\begin{eqnarray} \label{eq22}
\mathbf{t}_{N}^{\rm ske} \rightarrow \mathcal{T}_{\alpha}^{\rm ske}, \hspace*{4mm}  \text{a.s.},
\end{eqnarray}

\noindent with respect to the LWV-topology; this convergence has been already established in \cite[Examples 4.2 and 4.3]{LohrVoisinWinter2015} for a particular offspring distribution that lies in the domain of attraction of an $\alpha$-stable law. Then, to conclude with the proof of Proposition \ref{Prop1}, it is enough to show that,
as $N \rightarrow \infty$,
\begin{eqnarray}  \label{eq19}
\mathbf{t}_{N}^{\rm bra} \rightarrow  \mathcal{T}_{\alpha}^{\rm bra}, \hspace*{4mm}  \text{a.s.},
\end{eqnarray}

\noindent with respect to the LWV-topology. More precisely, \cite[Lemma 2.28]{LohrVoisinWinter2015}, \eqref{eq22} and \eqref{eq19} imply that $\mathbf{t}_{N}^{\rm mix} \rightarrow \mathcal{T}_{\alpha}^{\rm mix}$, as $N \rightarrow \infty$, a.s., with respect to the LWV-topology.

Note that $\mu_{N}^{{ \rm nod}}(\mathbf{t}_{N}) =1 \rightarrow \mu_{\alpha}(\mathcal{T}_{\alpha}) = 1$, as $N \rightarrow \infty$. By the definition of the LWV-topology (Definition \ref{Def2}), (\ref{eq19}) holds if for any $k \in \mathbb{N}$,
\begin{eqnarray}  \label{eq20}
(\llbracket \mathbf{U}_{k,N} \rrbracket, a_{N} \cdot r_{N}^{\rm gr}, \rho_{N},  \mathbf{U}_{k,N}, \nu_{N}^{\rm bra}) \rightarrow (\llbracket \mathbf{U}_{k} \rrbracket, r_{\alpha}, \rho_{\alpha}, \mathbf{U}_{k}, \nu_{\alpha}^{\rm bra}), \hspace*{4mm} \text{as} \hspace*{2mm} N \rightarrow \infty,
\end{eqnarray}

\noindent a.s., with respect to the Gromov-weak topology, where $\mathbf{U}_{k,N} = (U_{1,N}, \dots, U_{k,N})$ is a sequence of $k$ i.i.d.\ vertices with common distribution $\mu_{N}^{{\rm nod}}$ and $\mathbf{U}_{k} = (U_{1}, \dots, U_{k})$ is a sequence of $k$ i.i.d.\ points with common distribution $\mu_{\alpha}$; we tacitly understand that the pruning measure $\nu_{N}^{\rm bra}$ (resp.\ $ \nu_{\alpha}^{\rm bra}$) and the metric $a_{N} \cdot r_{N}^{\rm gr}$ (resp.\  $r_{\alpha})$ are restricted to the appropriate space.

We deal with the cases $\alpha =2$ and $\alpha \in (1,2)$ separately.

\paragraph{The Brownian case, $\alpha =2$.} In this case $\nu_{2}^{\rm bra} = \lambda_{2}$ is the length measure of the Brownian CRT.
 Suppose that we have shown that, for any $k \in \mathbb{N}$,
\begin{eqnarray} \label{eq21}
d_{\text{Pr}}^{a_{N} \cdot \mathbf{t}_{N}}(\nu_{N}^{{ \rm bra}} \upharpoonright_{\llbracket \mathbf{U}_{k,N} \rrbracket}, \nu_{N}^{{\rm ske}} \upharpoonright_{\llbracket \mathbf{U}_{k,N} \rrbracket}) \rightarrow 0, \hspace*{4mm} \text{as} \hspace*{2mm} N \rightarrow \infty,
\end{eqnarray}

\noindent almost surely. Therefore, for $\alpha =2$, (\ref{eq20}) follows from (\ref{eq22}) and \cite[Lemma 2.30]{LohrVoisinWinter2015}.

Let us now prove (\ref{eq21}). Denote by $B_{N}(v, \delta) \coloneqq \inf\{ u \in \mathbf{t}_{N}: a_{N} \cdot r_{N}^{\rm gr}(v, u) \leq \delta \}$ the closed ball in $a_{N} \cdot \mathbf{t}_{N}$ centred at $v \in \mathbf{t}_{N}$ and with radius $\delta \geq 0$. First, we consider the case $k =1$. For $\delta \geq 0$, we define
\begin{eqnarray}
F_{N}^{\rm bra}(\delta) \coloneqq \nu_{N}^{\rm bra}( B_{N}(\rho_{N}, \delta) \cap [\rho_{N}, U_{1,N}]) \hspace*{2mm} \text{and} \hspace*{2mm} F_{N}^{\rm ske}(\delta) \coloneqq \nu_{N}^{\rm ske}( B_{N}(\rho_{N}, \delta)  \cap [\rho_{N}, U_{1,N}]).
\end{eqnarray}

\noindent For $\ast \in \{{\rm bra}, {\rm ske}\}$, note that $F_{N}^{\ast}(\cdot)$ determines the measure $\nu_{N}^{\ast} \upharpoonright_{\llbracket \mathbf{U}_{k,N} \rrbracket}$ in the same way a distribution function determines a finite measure in $\mathbb{R}_{+}$. On the other hand, observe that
\begin{eqnarray}
 \sup_{\delta \geq 0} \left | F_{N}^{\rm bra}(\delta) -  F_{N}^{\rm ske}(\delta) \right| \leq \sup_{t \in [0,1]}  \left | b_{N} \sigma_{N}^{\uparrow}(t) -  a_{N} H_{N}((N-1)t) \right|,
\end{eqnarray}

\noindent where $(H_{N}((N-1)t), t \in [0,1])$ is the (normalize) height process associated to $\mathbf{t}_{N}$, and $(\sigma_{N}^{\uparrow}(t), t \in [0,1])$ is the process defined in (\ref{eq24}). By Theorem \ref{Theo3} and Theorem \ref{Theo4}, we known that both processes $(a_{N} H_{N}((N-1)t),  t \in [0,1])$ and $(b_{N}  \sigma_{N}^{\uparrow}(t),  t \in [0,1])$ converge to $2X_{2}$. Furthermore, these convergences hold uniformly on $[0,1]$. Thus, by the triangle inequality, we deduce that
\begin{eqnarray}
 \sup_{\delta \geq 0} \left | F_{N}^{\rm bra}(\delta) -  F_{N}^{\rm ske}(\delta) \right|  \rightarrow 0, \hspace*{4mm} \text{as} \hspace*{2mm} N \rightarrow \infty,
\end{eqnarray}
\noindent almost surely, which shows (\ref{eq21}) for $k =1$.

Suppose now that $k \geq 2$. For $1 < n \leq k$, we let $b_{n,N}$ be the branching point of $\mathbf{t}_{N}$ between $U_{n,N}$ and $\llbracket \mathbf{U}_{n-1,N} \rrbracket$, i.e., $b_{n,N} \in \llbracket \mathbf{U}_{n-1,N} \rrbracket$ such that $[\rho_{N}, U_{n,N}] \cap \llbracket \mathbf{U}_{n-1,N} \rrbracket = [\rho_{N}, b_{n,N}]$. Define, for $\delta \geq 0$ and $\ast \in \{{\rm bra}, {\rm ske}\}$,
\begin{eqnarray}
F_{N,1}^{\ast}(\delta) \coloneqq F_{N}^{\ast}(\delta) \hspace*{2mm} \text{and} \hspace*{2mm} F_{N,n}^{\ast}(\delta) \coloneqq \nu_{N}^{\ast}( B(b_{n,N}, \delta) \cap ]b_{n,N}, U_{n,N}]),
\end{eqnarray}

\noindent for $1 < n \leq k$. For $\ast \in \{{\rm bra}, {\rm ske}\}$, conditional on $\llbracket \mathbf{U}_{k,N} \rrbracket$, the vector $(F_{N,1}^{\ast}(\cdot), \dots, F_{N,k}^{\ast}(\cdot))$ determines the measure $\nu_{N}^{\ast} \upharpoonright_{\llbracket \mathbf{U}_{k,N} \rrbracket}$ in the same way as before. Then by a similar argument as the case $k=1$, one can show that
\begin{eqnarray}
\max_{1 \leq n \leq k} \sup_{\delta \geq 0} \left | F_{N,n}^{\rm bra}(\delta) -  F_{N,n}^{\rm ske}(\delta) \right| \rightarrow 0, \hspace*{4mm} \text{as} \hspace*{2mm} N \rightarrow \infty,
\end{eqnarray}
\noindent almost surely, which finishes the proof of (\ref{eq21}).

\paragraph{The stable case, $\alpha \in (1,2)$.} Let $p_{\alpha}$ be the canonical projection from $[0,1]$ to $\mathcal{T}_{\alpha}= [0,1] \setminus \sim_{\alpha}$, introduced in Section \ref{Sec13}, and let $(t_{i}, i \in \mathbb{N})$ be a sequence of independent uniform random variables on $[0,1]$. Assume they are also independent of $W_{N}^{\uparrow}$, $W_{N}^{\downarrow}$, $H_{N}$, $C_{N}$, $X_{\alpha}$ and $H_{\alpha}$. For $i \geq 1$, set $U_{i} = p_{\alpha}(t_{i})$ and observe that $(U_{i}, i \in \mathbb{N})$ is a sequence of i.i.d.\ random variables with ditribution $\mu_{\alpha}$. Define the projection $p_{N}$ from $\{0, 1/N, \dots, (N-1)/N\}$ onto the set of vertices of $\mathbf{t}_{N}$ by letting $p_{N}(i/N) = u_{N}(i)$, for $i \in \{0, 1, \dots, N-1\}$, where $u_{N}(i)$ is the $i$-th vertex of $\mathbf{t}_{N}$ in lexicographical order. Note that $(\lfloor t_{i} N \rfloor, i \in \mathbb{N})$ is a sequence of i.i.d.\ uniform integers on $\{0, 1, \dots, N-1\}$. For $i \geq 1$, set $t_{i}^{N} = \lfloor t_{i} N \rfloor/N$ and $U_{i,N} = p_{N}(t_{i}^{N}) = u_{N}(\lfloor t_{i} N \rfloor)$. Then $(U_{i, N}, i \in \mathbb{N})$ is a sequence of i.i.d.\ random variables with distribution $\mu_{N}^{{\rm nod}}$. Note that the sequence $(t_{i}^{N}, i \in \mathbb{N})$ converges to $(t_{i}, i \in \mathbb{N})$, almost surely. Furthermore, these sequences are independent of $a_{N} \cdot \mathbf{t}_{N}$ and $\mathcal{T}_{\alpha}$.

Let us prove (\ref{eq20}). Fix $k \in \mathbb{N}$. Recall that $\preccurlyeq$ denotes the genealogical order in $\mathbf{t}_{N}$ or $\mathcal{T}_{\alpha}$. (Here, we will consider the genealogical order in $\mathbf{t}_{N}$ with respect to the lexicographical order as defined in \eqref{eq8}.) Recall also that the support of $\mu_{\alpha}$ is the set of leaves of $\mathcal{T}_{\alpha}$ (\cite[Theorem 4.6]{Du2005}). Then, $U_{1}, \dots, U_{k}$ are leaves of $\mathcal{T}_{\alpha}$. Note that, for $i \in \{1, \dots, k\}$, $\{ s \in [0,1]: 0 \preccurlyeq s \preccurlyeq  t_{i} \hspace*{2mm} \text{and} \hspace*{2mm} \Delta_{s} > 0  \}$ is a countable set; see \cite[Theorem 4.6]{Du2005}. Moreover, for each fixed $\beta >0$, the set $\{ s \in [0,1]: 0 \preccurlyeq s \preccurlyeq t_{i} \hspace*{2mm} \text{and} \hspace*{2mm} \Delta_{s} > \beta  \}$ is finite. In particular, the set  $\{ s \in [0,1]: p_{\alpha}(s) \in \llbracket \mathbf{U}_{k} \rrbracket \hspace*{2mm} \text{and} \hspace*{2mm} \Delta_{s} > \beta\}$ is also finite. Here, $p_{\alpha}(s) \in \llbracket \mathbf{U}_{k} \rrbracket$ if there exists $i \in \{1, \dots, k\}$ such that $s \preccurlyeq  t_{i}$. Then, we write $\{ s \in [0,1]: p_{\alpha}(s) \in \llbracket \mathbf{U}_{k} \rrbracket \hspace*{2mm} \text{and} \hspace*{2mm} \Delta_{s} > \beta\} = \{s_{0}, \dots, s_{n_{\beta}}\}$ for some $n_{\beta} \in \mathbb{N}$ such that $p_{\alpha}(s_{i}) \in \llbracket \mathbf{U}_{k} \rrbracket$, for $i=1,\dots, n_{\beta}$. 

Properties of the Skorokhod topology entail that the jumps of $(W_{N}^{\uparrow}( \lfloor N t \rfloor ), t \in [0,1])$ converge to the jumps of $(X_{\alpha}(t), t \in [0,1])$, together with their locations (see for e.g., \cite[Proposition 2.7 in Chapter VI]{Jacod2003}). Hence it follows that for every $s \in \{s_{0}, \dots, s_{n_{\beta}}\}$ one can find $i_{N}(s) \in \{0, \dots, N-1\}$ such that the following two conditions hold for $N$ sufficiently large
\begin{enumerate}[label=(\textbf{\alph*})]
\item $\displaystyle \frac{i_{N}(s)}{N} \rightarrow s, \hspace*{2mm}  u_{N}(i_{N}(s)) \in \llbracket \mathbf{U}_{k, N} \rrbracket, \hspace*{2mm}  b_{N} \Delta W_{N}^{\uparrow}( i_{N}(s)+1 ) \rightarrow \Delta_{s}, \hspace*{2mm} \text{as} \hspace*{2mm} N \rightarrow \infty$, and \label{ProT1}
\item $\{ i_{N}(s_{0}), \cdots, i_{N}(s_{n_{\beta}}) \} = \{ r \in \{0, \dots, N-1\}:  u_{N}(r) \in \llbracket \mathbf{U}_{k, N} \rrbracket \hspace*{2mm} \text{and} \hspace*{2mm} b_{N}(c(u_{N}(r)) -1) > \beta  \}$. \label{ProT2}
\end{enumerate}
\noindent Recall that $c(u_{N}(m))-1 = \Delta W_{N}^{\uparrow}(m+1)$, for $m =0, \dots, N-1$, is the number of children minus one of the $m$-th vertex of $\mathbf{t}_{N}$ in lexicographical order.

For $i = 0, \dots, n_{\beta}$, we write $u_{i}^{N} =  p_{N}(i_{N}(s_{i})/N)$ and $u_{i} =  p_{\alpha}(s_{i})$. Define the finite measures $\vartheta_{N}^{n_{\beta}}(\cdot) \coloneqq b_{N} \sum_{i=0}^{n_{\beta}} (c(u_{i}^{N})-1) \delta_{u_{i}^{N}}(\cdot)$ and $\vartheta_{\alpha}^{n_{\beta}}(\cdot) \coloneqq \sum_{i=0}^{n_{\beta}} \Delta_{s_{i}} \delta_{u_{i}}(\cdot)$, on $\llbracket \mathbf{U}_{k,N} \rrbracket \subset \mathbf{t}_{n}$ and $\llbracket \mathbf{U}_{k} \rrbracket \subset \mathcal{T}_{\alpha}$, respectively. Suppose that we have shown that, for each fixed $\beta >0$, we have that, as $N \rightarrow \infty$,
\begin{eqnarray}  \label{eq25}
(\llbracket \mathbf{U}_{k,N} \rrbracket, a_{N} \cdot r_{N}^{{ \rm gr}}, \rho_{N}, \mathbf{U}_{k,N}, \vartheta_{N}^{n_{\beta}}) \rightarrow (\llbracket \mathbf{U}_{k} \rrbracket, r_{\alpha}, \rho_{\alpha}, \mathbf{U}_{k},  \vartheta_{\alpha}^{n_{\beta}}), \hspace*{2mm} \text{a.s.},
\end{eqnarray}
\noindent with respect to the Gromov-weak topology. 

Now, fix some $\varepsilon > 0$. Recall the definition of the Prohorov and pointed Gromov-Prohorov metrics given in \eqref{Proh1} and \eqref{Proh2}, respectively. First, note that we can choose $\beta >0$ small enough such that
\begin{eqnarray}   \label{eq26}
d_{\text{Pr}}^{\mathcal{T}_{\alpha}}(\nu_{\alpha}^{{ \rm bra}} \upharpoonright_{\llbracket \mathbf{U}_{k} \rrbracket}, \vartheta_{\alpha}^{n_{\beta}})  \leq \varepsilon /2.
\end{eqnarray}
\noindent On the other hand, by \eqref{eq6} in Lemma \ref{lemma7}, we can also choose  $\beta$ such that 
\begin{eqnarray}  \label{eq27}
\limsup_{N \rightarrow \infty}  d_{\text{Pr}}^{a_{N} \cdot \mathbf{t}_{N}}(\nu_{N}^{{ \rm bra}} \upharpoonright_{ \llbracket \mathbf{U}_{k,N} \rrbracket}, \vartheta_{N}^{n_{\beta}}) \leq \varepsilon /2.
\end{eqnarray}

\noindent Therefore, the triangle inequality together with (\ref{eq25}), (\ref{eq26}) and (\ref{eq27}) implies that
\begin{eqnarray}
\limsup_{N \rightarrow \infty}   d_{\text{pGP}}( (\llbracket \mathbf{U}_{k,N} \rrbracket, a_{N} \cdot r_{N}^{\rm gr}, \rho_{N},  \mathbf{U}_{k,N}, \nu_{N}^{\rm bra}),  (\llbracket \mathbf{U}_{k} \rrbracket, r_{\alpha}, \rho_{\alpha}, \mathbf{U}_{k}, \nu_{\alpha}^{\rm bra})) \leq \varepsilon, 
\end{eqnarray}

\noindent or equivalently, (\ref{eq20}). 

It only remains to prove (\ref{eq25}). We start by building a correspondence between $\mathcal{X}_{N} \coloneqq {\rm supp}(\vartheta_{N}^{n_{\beta}}) \cup \{\rho_{N}, U_{1,N}, \dots, U_{k,N} \}$ and $\mathcal{X} \coloneqq {\rm supp}(\vartheta_{\alpha}^{n_{\beta}}) \cup \{\rho_{\alpha}, U_{1}, \dots,U_{k} \}$ by letting
\begin{align}
\mathcal{R}_{N} \coloneqq \{ (u_{i}^{N}, u_{i}): \hspace*{1mm} u_{i}^{N} =  p_{N}(i_{N}(s_{i})/N) \hspace*{1mm} \text{and} \hspace*{1mm} u_{i} =  p_{\alpha}(s_{i}) \hspace*{1mm} \text{for} \hspace*{1mm}  i = 0, \dots, n_{\beta} \} \cup \mathcal{P}_{k,N},
\end{align}
\noindent where $\mathcal{P}_{k,N} \coloneqq \{ (\rho_{N}, \rho_{\alpha}), (U_{1,N}, U_{1}),  \dots, (U_{k,N}, U_{k})\}$. Recall that $\rho_{N} = p_{N}(0)$, $\rho_{\alpha} = p_{\alpha}(0)$, $U_{i,N}= p_{N}(t_{i}^{N})$ and $U_{i} = p_{\alpha}(t_{i})$, for $i=1, \dots, k$. The distortion of $\mathcal{R}_{N}$ is defined by
\begin{eqnarray}
{\rm dist }( \mathcal{R}_{N}) = \sup \{ | a_{N} \cdot r_{N}(x_{1},y_{1}) - r_{\alpha}(x_{2},y_{2})| : (x_{1}, x_{2}), (y_{1}, y_{2}) \in \mathcal{R}_{N} \}.
\end{eqnarray}

\noindent Let $E = \mathcal{X}_{N} \uplus \mathcal{X}$ be the disjoint union of $\mathcal{X}_{N}$ and $\mathcal{X}$. Define a metric $r_{E}$ on $E$ by letting
\begin{eqnarray}
r_{E}(x,y) = \left\{ \begin{array}{ll}
              a_{N} \cdot r_{N}(x,y) & \mbox{if} \, \, \, \, x,y \in \mathcal{X}_{N}, \\
             r_{\alpha}(x,y)  & \mbox{if} \, \, \, \,  x,y \in \mathcal{X},  \\
             \inf \{ a_{N} \cdot r_{N}(x,x^{\prime}) + r_{\alpha}(y,y^{\prime}) + \frac{1}{2} {\rm dist }( \mathcal{R}_{N}): (x^{\prime}, y^{\prime}) \in \mathcal{R}_{N} \} + N^{-1}\mathbf{1}_{\{{\rm dist }( \mathcal{R}_{N})=0 \}} & \mbox{if} \, \, \, \,  x \in \mathcal{X}_{N}, y \in \mathcal{X},  \\
              \end{array}
    \right.
 \hspace*{4mm}
\end{eqnarray}
\noindent for $x,y \in E$. We leave it to the reader to check that $r^{E}$ is, indeed, a metric; the extra term $N^{-1}\mathbf{1}_{\{{\rm dist }( \mathcal{R}_{N})=0 \}}$ is to avoid having a pseudo-metric when ${\rm dist }( \mathcal{R}_{N})=0$. In particular, $r_{E}(x,y) = \frac{1}{2} {\rm dist }( \mathcal{R}_{N}) + N^{-1}\mathbf{1}_{\{{\rm dist }( \mathcal{R}_{N})=0 \}}$, for $(x,y) \in \mathcal{R}_{N}$. On the other hand, Theorem \ref{Theo3} implies that
\begin{eqnarray}
{\rm dist }( \mathcal{R}_{N}) \leq 4 \sup_{t \in [0,1]} \left | a_{N} H_{N}((N-1)t) - H_{\alpha}(t) \right| \rightarrow 0, \hspace*{2mm} \text{as} \hspace*{2mm} N \rightarrow \infty, \hspace*{2mm} \text{a.s.} 
\end{eqnarray}
\noindent Hence the above observation together with \ref{ProT1} and \ref{ProT2} shows that, as $N \rightarrow \infty$,
\begin{eqnarray}
r_{E}(\rho_{N}, \rho_{\alpha}) \rightarrow 0, \hspace*{2mm} r_{E}(U_{i,N},U_{i}) \rightarrow 0, \hspace*{2mm} \text{for} \hspace*{2mm} i =1,\dots, k,
\end{eqnarray}
\noindent and
\begin{eqnarray}
d_{\rm Pr}^{(E, r_{E})}(\vartheta_{N}^{n_{\beta}}, \vartheta_{\alpha}^{n_{\beta}} ) \rightarrow 0,
\end{eqnarray}
\noindent almost surely. Indeed, the third convergence is a consequence of \cite[Example 2.1 in Chapter 1]{Bi1999} and \cite[Theorem 3.1 in Chapter 3]{Et1986}. Therefore, (\ref{eq25}) follows from the characterisation of convergence in the (pointed) Gromov-weak topology via the pointed Gromov-Prokhorov metric (\cite[Proposition 2.6]{LohrVoisinWinter2015}). 

This concludes the proof of Theorem \ref{NewTheo}.
\end{proof}

\begin{appendices}
\section{Continuity of some functionals of excursions} \label{Append1}

We write $\mathbb{D}([0,1], \mathbb{R})$ for the space of real-valued c\`adl\`ag functions on $[0,1]$ equipped with the Skorohod ${\rm J}_{1}$ topology; see e.g., \cite[Chapter 3]{Bi1999} or \cite[Chapter VI]{Jacod2003}. If $g \in \mathbb{D}([0,1], \mathbb{R})$, we set $\Delta g(t) \coloneqq g(t) - g(t-)$, for $t \in [0,1]$, with the convention $g(0) = g(0-)$. Let $J(g) = \{ t \in [0,1]: |\Delta g(t)| >0 \}$ be the set of all jump times of $g$. Observe that $J(g) \subseteq (0,1)$ and recall that $J(g)$ is at most countable. 

Through this section, we fix a function $f \in  \mathbb{D}([0,1], \mathbb{R})$ such that $f(0) =f(1) = f(1-) =0$ (with the convention $f(0-)= 0$), $f(t) > 0$ and $\Delta f(t) \geq 0$ for every $t \in (0,1)$, and satisfying the following properties: 

\begin{enumerate}[label=(\textbf{H.\arabic*})]
\item If $0 \leq s \leq t \leq 1$, there exists at most one value $r \in (s,t)$ such that $f(r) = \inf_{u \in [s,t]} f(u)$ (we say that local minima of $f$ are distinct); \label{H1}
\item If $t \in (0,1)$ is such that $\Delta f(t) > 0$, then $\inf_{u \in [t,t+\varepsilon]} f(u) < f(t)$ for all $\varepsilon \in (0, 1-t]$; \label{H2}
\item If $t \in (0,1)$ is such that $\Delta f(t) > 0$, then $\inf_{u \in [t-\varepsilon,t]} f(u) < f(t-)$ for all $\varepsilon \in (0, t]$; \label{H3}
\item Suppose that $f$ attains a local minimum at $t \in (0,1)$ (in particular, $\Delta f(t) = 0$ by  \ref{H3}). Let $s = \sup \{r \in [0,t]: f(r) < f(t) \}$. Then $\Delta f(s) > 0$ and $f(s-) <f(t)$. Note that then $f(s) > f(t)$ by \ref{H2}. \label{H4}
\end{enumerate}

For every $t \in [0,1]$, we set
\begin{align}
\ell(t) \coloneqq \inf\{ s \in (0, 1-t]: f(t+s) - f(t) < 0 \} 
\end{align}
\noindent and 
\begin{align}
\ell(t-) \coloneqq \inf\{ s \in (0, 1-t]: f(t+s) - f(t-) < 0 \},
\end{align}
\noindent with the convention $\ell(1) = \ell(1-) =1$. 
Since $f(1-)=f(1) = 0$ and $f(s) > 0$ for every $s \in (0,1)$, we have that $\ell(t), \ell(t-) >0$, for all $t \in [0,1]$. 

\begin{lemma}  \label{lemma1}
We have that $\ell(0) = \ell(0-) =1$. Moreover, for all $t \in (0,1) \setminus J(f)$, we have that $\ell(t) = \ell(t-) = 0$.
\end{lemma}

\begin{proof}
Clearly, $\ell(t) = \ell(t-)$ for all $t \in [0,1] \setminus J(f)$. Moreover, $\ell(0) =1$, which proves our first claim. Next, we prove the second claim by contradiction. Suppose $\ell(t) \in (0,1-t]$. Since $\Delta f(s) \geq 0$ for every $s \in (0,1)$, it must be that $t+\ell(t) \in (0,1) \setminus J(f)$, meaning $t+\ell(t)$ is a continuity point of $f$. 
As $[0,1]\setminus J(f)$ is dense in $[0,1]$, for any $\varepsilon > 0$,  there exists $t^{\prime} \in (0, \ell(t))$ such that $f(t+t^{\prime}) < f(t+\ell(t)) + \varepsilon$.  By choosing $\varepsilon = f(t) - f(t+\ell(t)) > 0$ and $t^{\prime}$ accordingly, we find $f(t+t^{\prime}) < f(t)$, which establishes the contradiction.
\end{proof}

\begin{lemma} \label{lemma2}
For $t \in J(f)$, we have that $\ell(t) = 0$ and that $\ell(t-) =  \tilde{\ell}(t-)$, where 
\begin{align}
\tilde{\ell}(t-) & \coloneqq \inf\{ s \in (0, 1-t]: f(t+s) = f(t-) \}  = \inf\{ s \in (t, 1]: f(s) = f(t-) \}-t.
\end{align}
\end{lemma}

\begin{proof}
Recall that $J(f) \subseteq (0,1)$. Note that $\ell(t) = 0$ is a consequence of \ref{H2}. Next, observe that, for $t \in J(f)$,
\begin{align} \label{eq4}
f(t+\tilde{\ell}(t-)) = f(t-) < \inf_{s \in [0, \tilde{\ell}(t-) -\varepsilon]} f(t+s), \quad \text{for all} \, \, \varepsilon \in (0, \tilde{\ell}(t-)]
\end{align}
\noindent On the other hand, we also have that, for $t \in J(f)$,
\begin{align} \label{eq5}
\inf_{s \in [0, \varepsilon]} f(t + \tilde{\ell}(t-) +s)< f(t+\tilde{\ell}(t-)) = f(t-), \quad \text{for all} \, \, \varepsilon \in (0, 1-t -\tilde{\ell}(t-) ].
\end{align}
\noindent To see this last bound, note that otherwise $t + \tilde{\ell}(t-)$ would be a time of a local minimum of $f$, which would contradict \ref{H4}. Then, it follows from \eqref{eq4} and \eqref{eq5} that $\ell(t-) =  \tilde{\ell}(t-)$.
\end{proof}

For $a \in [0,1]$, we write $\mathbb{D}([0,a], [0,1])$ for the space of real-valued c\`adl\`ag functions from $[0,a]$ to $[0,1]$ equipped with the Skorohod ${\rm J}_{1}$ topology. 

\begin{lemma} \label{lemma3}
Let $(f_{n})_{n\in\mathbb{N}}$ be a sequence of functions in $\mathbb{D}([0,1], \mathbb{R})$ that converges to  $f \in \mathbb{D}([0,1], \mathbb{R})$ as above. Let $(\theta_{n})_{n\in\mathbb{N}}$ be a sequence of functions in $\mathbb{D}([0,1], [0,1])$ such that 
\begin{align} \label{eq28}
\sup_{s \in [0,1]} | \theta_{n}(s) - \theta(s)| \rightarrow 0, \hspace*{4mm} \text{as} \hspace*{2mm}  n\rightarrow \infty,
\end{align}
\noindent where $\theta(s) \coloneqq (t+s) \wedge 1$, for $s,t \in [0,1]$.
\noindent Then,
\begin{align} \label{eq2}
(f_{n}(\theta_{n}(s)), s \in [0,1])  \rightarrow (f(\theta(s)), s \in [0,1]), \quad \text{in} \hspace*{2mm}  \mathbb{D}([0,1], \mathbb{R}), \hspace*{2mm} \text{as} \hspace*{2mm}  n\rightarrow \infty.
\end{align}

Let $(t_{n})_{n \in \mathbb{N}}$ and $(r_{n})_{n \in \mathbb{N}}$ be any two sequences of real numbers in $[0,1]$ such that $t_{n} \rightarrow t \in [0,1]$ and $r_{n} \downarrow 0 \in [0,1]$, as $n\rightarrow \infty$. For $n \in \mathbb{N}$, define
\begin{align}
\ell_{n}(t_{n}) \coloneqq \inf  \{ s \in [r_{n}, 1]: f_{n}(\theta_{n}(s)) - f_{n}(t_{n})  <0   \}.
\end{align}
\noindent Then, 
\begin{enumerate}[label=(\textbf{\roman*})]
\item If $f_{n}(t_{n}) \rightarrow f(t-)$ as  $n\rightarrow \infty$, then 
$\ell_{n}(t_{n}) \rightarrow \ell(t-)$; \label{lemma3C1}
\item  If $f_{n}(t_{n}) \rightarrow f(t)$ as  $n\rightarrow \infty$, then 
$\ell_{n}(t_{n}) \rightarrow \ell(t)$.  \label{lemma3C2}
\end{enumerate}
\end{lemma}

\begin{proof}
First, we prove \eqref{eq2}. It follows from \cite[Theorem 3.1]{Whitt1980} and \eqref{eq28} that $(f_{n}(\theta_{n}(s)), s \in [0,1-t])_{n \in \mathbb{N}}$ converges to $(f(\theta(s)), s \in [0,1-t])$ in $\mathbb{D}([0,1-t], [0,1])$. Then, there exists a strictly increasing continuous function $\lambda_{1}:[0,1-t] \rightarrow [0,1-t]$ such that $\lambda_{1}(0) =0$, $\lambda_{1}(1-t) = 1-t$ and
\begin{align} \label{eq3}
\sup_{s \in [0,1-t] } |f_{n}(\theta_{n}(s)) - f(\theta(\lambda_{1}(s)))|  \rightarrow 0, \hspace*{4mm} \text{as} \hspace*{2mm}  n\rightarrow \infty.
\end{align}
\noindent Define $\lambda_{2}(s) = \lambda_{1}(s)$, for $s \in [0,1-t]$ and $\lambda_{2}(s) = s$, for $s \in (1-t, 1]$. In particular, $\lambda_{2}$ is a strictly continuous function from $[0,1]$ onto itself. On the other hand, \eqref{eq28} implies that
\begin{align} \label{eq29}
\sup_{s \in [1-t,1]} | f_{n}(\theta_{n}(s))| \rightarrow 0, \hspace*{4mm} \text{as} \hspace*{2mm}  n\rightarrow \infty.
\end{align}
\noindent (Recall that $f(1) =0$). Therefore, it is not difficult to conclude from \eqref{eq3} and \eqref{eq29}  that 
\begin{align}
\sup_{s \in [0,1] } |f_{n}(\theta_{n}(s)) - f(\theta(\lambda_{2}(s)))|  \rightarrow 0, \hspace*{4mm} \text{as} \hspace*{2mm}  n\rightarrow \infty,
\end{align}
\noindent which implies \eqref{eq2}. 

Next we prove \ref{lemma3C1}. Note that $\ell(t-) = \inf  \{ s \in (0, 1]: f(\theta(s)) - f(t-)  <0   \}$. 

If $t=1$, then $\limsup_{n \rightarrow \infty} \ell_{n}(t_{n}) \leq 1 = \ell(t-)$. Suppose that $t \in [0,1)$. As $[0,1]\setminus J(f)$ is dense in $[0,1]$, for any $\delta \in (0, 1-t-\ell(t-)]$, there exists a continuity point $t^{\prime} \in [\ell(t-),\ell(t-) + \delta]$ of $f(\theta(\cdot))$ such that $f(\theta(t^{\prime})) - f(t-) < 0$. But then, since we have assumed that $f_{n}(t_{n}) \rightarrow f(t-)$, we conclude from \eqref{eq2} that $f_{n}(\theta_{n}(t^{\prime}))- f_{n}(t_{n}) \rightarrow f(\theta(t^{\prime})) - f(t-)$, as $n \rightarrow \infty$. Thus,  $f_{n}(\theta_{n}(t^{\prime}))- f_{n}(t_{n})  < 0$ for $n$ large enough. We deduce that $\limsup_{n \rightarrow \infty} \ell_{n}(t_{n}) \leq \ell(t-)$ (note that we also require $n$ to be large enough such that  $t^{\prime} > r_{n}$). 

If $t \in (0,1) \setminus J(f)$, then $\ell(t-) = 0$ by Lemma \ref{lemma1}, which implies \ref{lemma3C1}. Suppose that $t \in J(f)$ or $t =0,1$. Then, it follows that for any $\delta \in (0, \ell(t-))$, and for any $s \in (0, \ell(t-) - \delta]$, $f(\theta(s)) - f(t-) > 0$ and $f(\theta(s-)) - f(t-) > 0$. Hence, there exists $\eta > 0$ such that 
\begin{align}
d(f(\theta(s))- f(-t), (-\infty, 0]) > \eta \quad \text{and}  \quad d(f(\theta(s-))- f(t-), (-\infty, 0]) > \eta,
\end{align}
\noindent for any $s \in (0, \ell(t-) - \delta]$; here $d$ denotes the Euclidean distance in $\mathbb{R}$. (Indeed, this follows from Lemma \ref{lemma2} for $t \in J(f)$ and for $t=0,1$ by recalling that $f(0-)=f(0) =f(1-)=f(1) = 0$ and $f(s) > 0$ for every $s \in (0,1)$.) But then, by \eqref{eq2}, there exists $N \in \mathbb{N}$ such that
\begin{align}
d(f_{n}(\theta_{n}(s))- f_{n}(t_{n}), (-\infty, 0]) > \eta \quad \text{and}  \quad d(f_{n}(\theta_{n}(s-))- f_{n}(t_{n}), (-\infty, 0]) > \eta,
\end{align}
\noindent for any $s \in (0, \ell(t-) - \delta]$ and $n \geq N$. The above implies that $\ell_{n}(t_{n}) \geq (\ell(t-)-\delta) \vee r_{n}$ for $n \geq N$ and we deduce that $\liminf_{n \rightarrow \infty} \ell_{n}(t_{n}) \geq \ell(t-)-\delta$. Therefore, $\ell_{n}(t_{n}) \rightarrow \ell(t-)$, which concludes the proof of \ref{lemma3C1}.

Finally, the proof of \ref{lemma3C2} is similar, and the details are left to the reader. 
\end{proof}
\end{appendices}

\paragraph{Acknowledgements.}
The authors thank the anonymous reviewer for the careful revision of our manuscript. Their insightful comments greatly improved its overall presentation and helped clarify several obscure parts of an earlier version.


\providecommand{\bysame}{\leavevmode\hbox to3em{\hrulefill}\thinspace}
\providecommand{\MR}{\relax\ifhmode\unskip\space\fi MR }
\providecommand{\MRhref}[2]{%
  \href{http://www.ams.org/mathscinet-getitem?mr=#1}{#2}
}
\providecommand{\href}[2]{#2}

\end{document}